\documentclass[11pt,a4paper,reqno]{amsart}

\usepackage[utf8]{inputenc}
\usepackage[T1]{fontenc}
\usepackage[margin=3cm]{geometry}
\usepackage{amsmath}
\usepackage{amsthm}
\usepackage{amssymb}
\usepackage{amsopn}
\usepackage{enumitem}
\usepackage{xcolor}
\usepackage{nccmath}

\usepackage{mathtools}
\usepackage{doi}
\usepackage{xfrac}
\usepackage{graphicx}
\usepackage{setspace}

\usepackage{bm}
\usepackage[bb=dsserif]{mathalpha}

\usepackage{relsize}
\usepackage[labelformat=simple]{subcaption}

\usepackage[normalem]{ulem}
\usepackage{mathrsfs}

\makeatletter
\def\paragraph{\@startsection{paragraph}{4}%
  \z@\z@{-\fontdimen2\font}%
  {\normalfont\bfseries}}
\makeatother

\theoremstyle{plain}
\newtheorem{theorem}{Theorem}[section]
\newtheorem{prop}[theorem]{Proposition}

\newtheorem{lemma}[theorem]{Lemma}

\newtheorem{mytheorem}{Theorem}

\theoremstyle{definition}
\newtheorem{definition}[theorem]{Definition}

\theoremstyle{remark}
\newtheorem{remark}[theorem]{Remark}

\numberwithin{equation}{section}

\DeclareMathOperator{\dist}{dist}

\DeclareMathOperator{\supp}{supp}

\DeclareFontFamily{OT1}{pzc}{}
\DeclareFontShape{OT1}{pzc}{m}{it}{<-> s * [1.40] pzcmi7t}{}
\DeclareMathAlphabet{\mathpzc}{OT1}{pzc}{m}{it}

\newcommand{\Wc}{\mathcal{W}}
\newcommand{\Vsc}{\mathscr{V}}
\newcommand{\Vscs}{\mathscr{V}_{\rm s}}
\newcommand{\Ss}{\mathcal{S}}
\renewcommand{\Mc}{\ensuremath{\mathcal{M}}}

\newcommand{\R}{\mathbb{R}}

\newcommand{\Rc}{\ensuremath{\mathcal{R}}}

\newcommand{\Hcc}{\ensuremath{\dot{\mathscr{H}}^1}}
\newcommand{\La}{\Lambda}

\newcommand{\Z}{\mathbb{Z}}
\newcommand{\C}{\mathbb{C}}
\newcommand{\N}{\mathbb{N}}

\usepackage{color}

\definecolor{fgreen}{HTML}{2ECC40}

\newcommand{\corr}{}

\newcommand{\<}[2]{\ensuremath{\langle #1,\,#2\rangle}}

\newcommand{\E}{\mathcal{E}}

\newcommand{\B}{\mathcal{B}}

\newcommand{\G}{\mathcal{G}}

\makeatletter
\def\subsubsection{\@startsection{subsubsection}{3}%
  \z@{.5\linespacing\@plus.7\linespacing}{-.5em}%
  {\normalfont\bfseries}}

\begin{document}

\title{Atomistic modelling of near-crack-tip plasticity}


\author{Maciej Buze}
\address{School of Mathematics\\
  Cardiff University\\
  Senghennydd Road\\
  Cardiff\\
  CF24 4AG\\
  United Kingdom}
\email[M.~Buze]{buzem@cardiff.ac.uk}

\thanks{This work was supported by EPSRC, grant no. EP/S028870/1.}



\date{\today}

\dedicatory{}

\begin{abstract}
An atomistic model of near-crack-tip plasticity on a square lattice under anti-plane shear kinematics is formulated and studied. The model is based upon a new geometric and functional framework of {\it a lattice manifold complex}, which ensures that the crack surface is fully taken into account, while preserving the crucial notion of duality. As a result, existence of locally stable equilibrium configurations containing both a crack opening and dislocations is established. Notably, with the boundary in the form of a crack surface accounted for, no minimum separation between a dislocation core and the crack surface or the crack tip is required.

The work presented here constitutes a foundation for several further studies aiming to put the phenomenon of near-crack-tip plasticity on a rigorous footing. 
\end{abstract}
\maketitle



\section{Introduction}\label{sec:intro}
In a cracked crystalline body, the term {\it near-crack-tip plasticity} refers to the phenomenon of atoms rearranging themselves in the vicinity of the crack tip due to stresses accumulated therein \cite{SJ12}. This rearrangement most prominently comes in the form of topological defects known as dislocations, which are carriers of plastic (irreversible) deformation \cite{hirth-lothe}. With such deformations around the crack tip potentially having the capacity to shield the material from further crack propagation \cite{majumdar1981crack,majumdar1983griffith,ZYLS10}, a proper understanding of mechanisms involved in near-crack-tip plasticity is of paramount importance in the context of structural integrity of materials used for engineering purposes \cite{SDIM06}.

With cracks and dislocations initiating and propagating via primarily atomistic mechanisms \cite{hirth-lothe,bulatov2006computer,shimada2015breakdown,Bitzek2015}, the modelling of near-crack-tip plasticity should ideally take the atomistic scale correctly into account. This poses a major challenge, as the inherently nonlinear nature of interactions between atoms renders the interplay between defects in a crystal a highly complex phenomenon. In particular, there currently does not exist a mathematical theory providing a bespoke theoretical underpinning for the many computational models employed in practice, such as \cite{Machov__2009,CHENG20123345,YAMAKOV201435,RAJAN201618,berton2019atomistic}.

The primary reason for the current lack of such a theory is the sheer complexity of the processes involved and the intrinsically multiscale nature of the problem. On a physical level, it is, for instance, known that in certain regimes dislocations can be {\it emitted} from the crack tip \cite{Rice1992,Bitzek2015}, while there are also experimentally verified regimes in which {\it plastic-free zones} exist just ahead of the crack tip \cite{horton1982tem}, meaning that the dislocations can be effectively pushed away from the crack tip.

On a mathematical level, while there is a wealth of recent work about atomistic modelling of materials and in particular atomistic approaches to defects in crystals (c.f., among many other, \cite{BDMG99,AO05,BLO06,P07,BC07,E-Ming,ADLGP14,2013-disl,EOS2016}), many of the mathematical techniques employed rely at least in part on exploiting crystal symmetries. This renders most of them fundamentally inadequate for the study of near-crack-tip plasticity, as the crack breaks the translational symmetry, meaning that the domain under study is discrete and spatially inhomogeneous. Such a setup has so far only been considered in the context of atomistic crack propagation in \cite{2018-antiplanecrack,2019-antiplanecrack} with the crystalline material considered there implicitly assumed to be perfectly brittle, meaning no plastic deformation occurring around the crack tip was studied. To the best of author's knowledge, there currently does not exist any mathematical literature concerning an atomistic approach to modelling near-crack-tip plasticity.

This paper concerns developing a mathematical framework describing atomistic near-crack-tip plasticity in the setup of a square lattice under anti-plane shear kinematics, with atoms  assumed to interact via a nearest neighbour pair potential and the crack for simplicity is assumed to be stationary, with interactions across the crack surface disregarded. The framework is based upon the theory of discrete dislocations in crystals developed in \cite{AO05} and expanded in the case of screw dislocations in \cite{H17}, and introduces the idea of {\it a lattice manifold complex}, which formalises the concept of a discrete complex square root manifold employed in \cite{2018-antiplanecrack}.

The lattice manifold complex entails a geometric and functional setup in which a variational approach can be employed to prove existence and provide characterisation of {\it near-crack-tip plasticity} equilibrium configurations -- locally stable equilibria in which both a crack opening and dislocations are present. In the process the duality of the lattice manifold complex is exploited, which naturally leads to discussing a lattice Green's function with {\it zero Dirichlet boundary condition} along the crack surface. Such a Green's function is given a careful characterisation, which is a result also of independent interest. 

Notably, with the present formulation fully accounting for the boundary of the domain in the form of crack surface, it is not necessary to impose a common restriction in the form of a minimum separation distance between dislocations and the boundary, thus going beyond the regime considered in \cite{H17}.   With a cracked crystal representing an extreme case of a non-convex domain, the approach presented further paves the way for treating general non-convex domains, thus further extending the results in \cite{H17}.

This paper lays foundation for future study of atomistic near-crack-tip plasticity in a variety of contexts, including an on-going study \cite{BvM} aiming to rigorously address the upscaling of the model and the existence of a plastic-free zone in a mesoscopic description, which will provide a mathematical underpinning for the transmission electron microscopy (TEM) experiments in \cite{horton1982tem}. A separate clear future direction is to extend the framework to the case of a moving crack tip, thus allowing a rigorous study of the influence dislocations exert on crack propagation.

\subsection{ {\corr Outline of the main results}}\label{sec:out-res}
\begin{figure}[!htbp]
  \begin{subfigure}[t]{.48\textwidth}
    \centering
    \includegraphics[width=\linewidth]{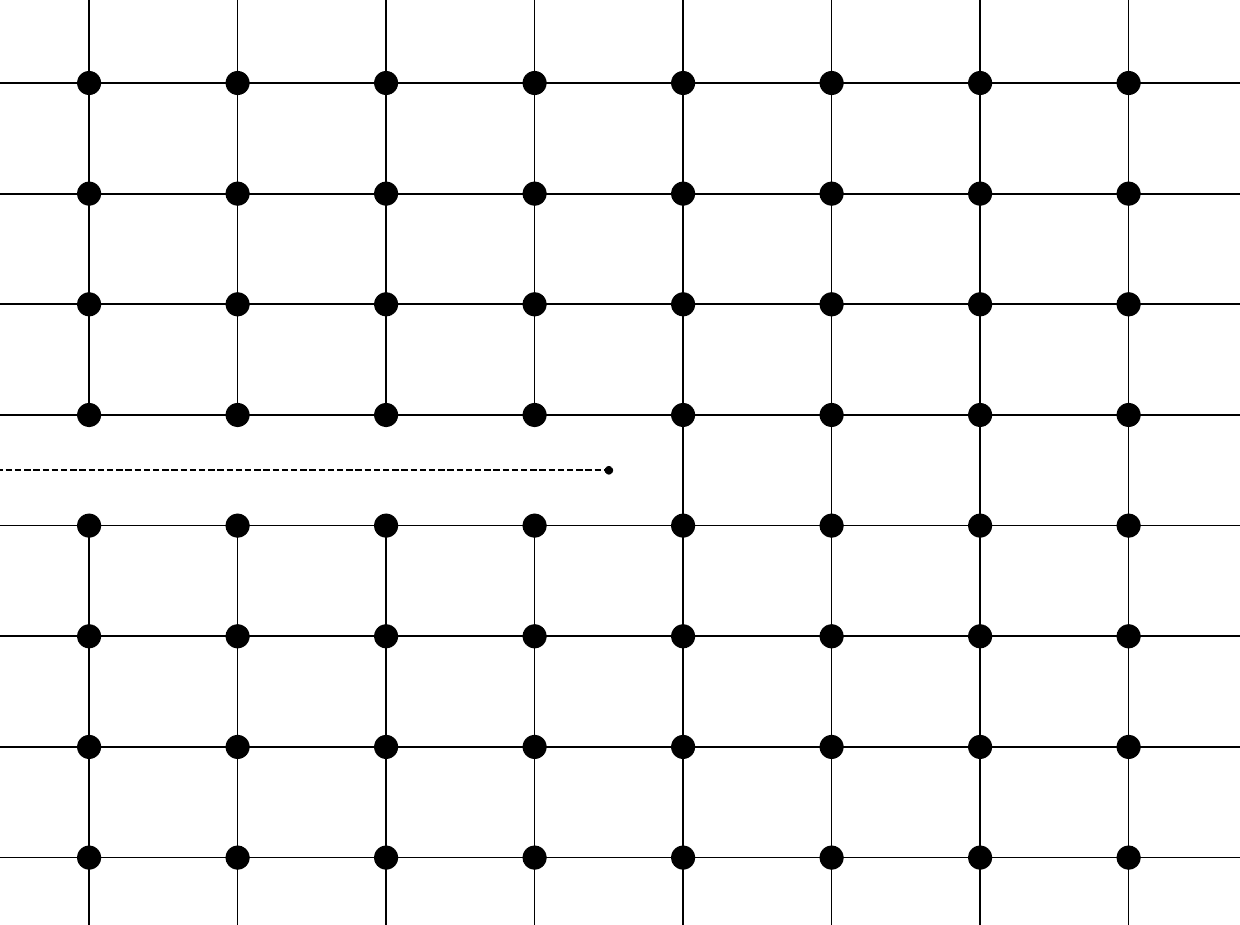}
    \caption{$\,$}\label{fig::intro1a}
  \end{subfigure}
	\quad
  \begin{subfigure}[t]{.48\textwidth}
    \centering
    \includegraphics[width=\linewidth]{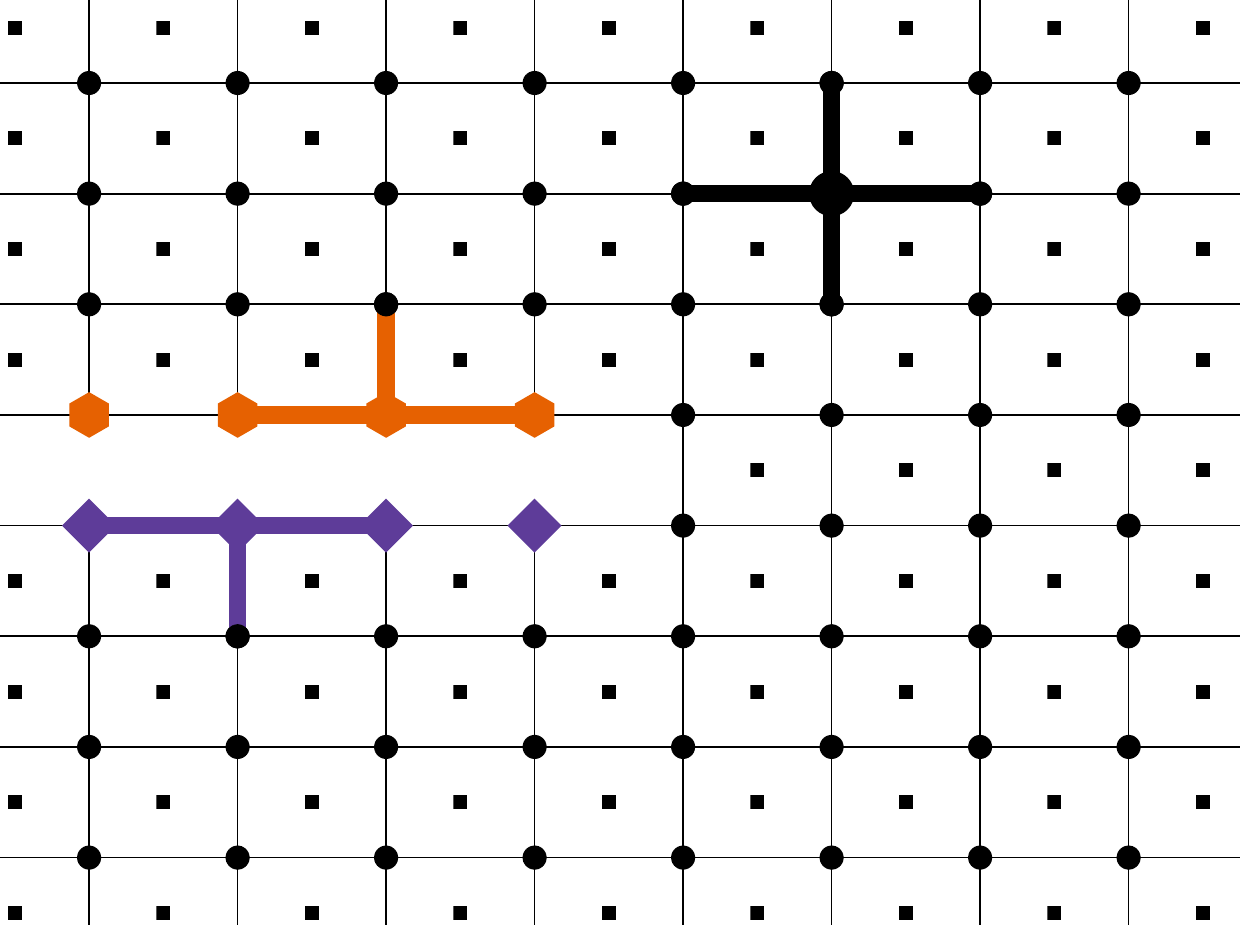}
    \caption{$\,$}\label{fig::intro1b}
  \end{subfigure}
  \captionsetup{width=0.95\linewidth}
  \caption{\corr A cracked crystal material $\Lambda$. The dotted line in \subref{fig::intro1a} is the crack surface $\Gamma_0$ along which the interactions are erased. In \subref{fig::intro1b} purple diamonds are atoms in $\Gamma_-$ and orange hexagons are atoms in $\Gamma_+$. A set of interactions of representative atoms are highlighted. The small squares in \subref{fig::intro1b} represent the dual lattice $\Lambda^*$. The erased bonds across the crack surface prevent the basic notion of duality from being applicable at the crack surface $\Gamma$.}
\label{fig:intro1}
\end{figure}
{\corr The focus of the paper is on a crystalline material modelled as a shifted square lattice $\Lambda = \Z^2 - \left(\tfrac{1}{2},\tfrac{1}{2}\right)$ subject to anti-plane displacements ${y\,\colon\,\Lambda \to \R}$ and under a nearest-neighbour interaction law between the atoms. The material is assumed to be cracked with the crack surface $\Gamma = \Gamma_+ \cup \Gamma_-$ where ${\Gamma_{\pm} = \{ (l_1,l_2) \in \Lambda\,\mid\,l_1 < 0,\; l_2 = \pm \frac{1}{2}\}}$. This is encoded in the setup by modifying the set of interactions for atoms at the crack surface, namely
\[
\Rc(l) = \begin{cases} \{e_1, -e_1, e_2,-e_2\}\quad&\text{ if } l \not\in\Gamma,\\
\{e_1, - e_1, \pm e_2\}\quad&\text{ if } l \in \Gamma_{\pm}, \end{cases} 
\]
where $e_1 = (1,0), e_2 = (0,1)$ form the canonical basis of $\R^2$. Figure~\ref{fig:intro1} provides some visual intuition of this setup. 

The atomistic energy considered is, formally (subject to an appropriate renormalisation, see Section~\ref{sec:atom_model}), of the form
\[
E(y) = \sum_{l \in \Lambda} \sum_{\rho \in \Rc(l)} \psi\big( y(l+\rho) - y(l)\big),
\]
where  $\psi(x) = \frac{\lambda}{2}\dist(x,\Z)^2$ for some $\lambda > 0$. The periodicity of $\psi$ encodes the fact that $\Lambda$ is to be considered a top-down view of a three-dimensional crystal and each $m \in \Lambda$ in fact represents a column of atoms. 

The phenomenon of near-crack-tip plasticity is discussed in terms of considering $m$ screw dislocations with cores at $\{x_1,\dots,x_m\} \subset \Lambda^*$, where $\Lambda^* = \Z^2$ is the dual to $\Lambda$, with each $l^* \in \Lambda^*$ identified with a square (a closed loop of four atoms) in $\Lambda$. The problem of preserving duality near the crack surface $\Gamma$ will be introduced in Section~\ref{sec:int-ideas}. The full details of the atomistic model are presented in Section~\ref{sec:atom_model}.

The first main result is as follows (see Section~\ref{sec:existence_k0} for the full treatment).
\begin{mytheorem}[Theorem~\ref{thm:K0}]\label{thm:A}
There exists a locally stable equilibrium (notion made precise in Definition~\ref{def:stability}) configuration $y_{\rm s} \,\colon\,\Lambda \to \R$ containing $m$ dislocations, provided that cores satisfy the minimum separation property (see Figure~\ref{fig:intro2} for some visual intuition and the notion is made precise in Section~\ref{sec:dis_conf}). Notably, no such separation condition between dislocation cores and the crack surface $\Gamma$ is needed.
\end{mytheorem}

The construction of $y_{\rm s}$ is explicit and uses the well-known duality between a point source on the dual lattice and a dislocation core on the original lattice \cite{AO05,H17}. Crucially, a homogeneous Neumann boundary condition related to the crack surface $\Gamma$ on the original lattice becomes a homogeneous Dirichlet boundary condition on the dual. The following is proven (full details are presented in Section~\ref{sec:G}), which is of independent of interest.

\begin{figure}[!htbp]
  \begin{subfigure}[t]{.48\textwidth}
    \centering
    \includegraphics[width=\linewidth]{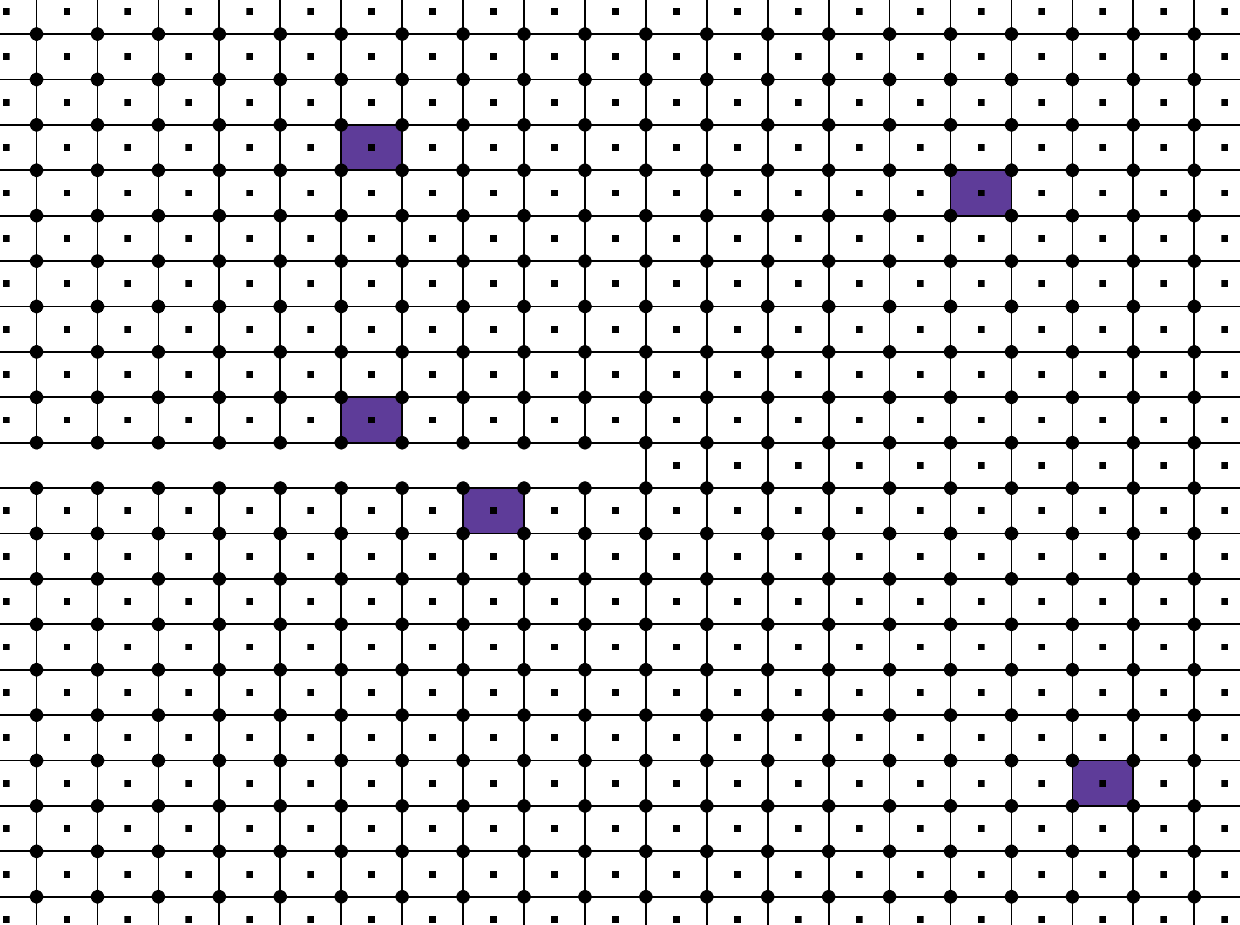}
    \caption{$\,$}\label{fig::intro2a}
  \end{subfigure}
	\quad
  \begin{subfigure}[t]{.48\textwidth}
    \centering
    \includegraphics[width=\linewidth]{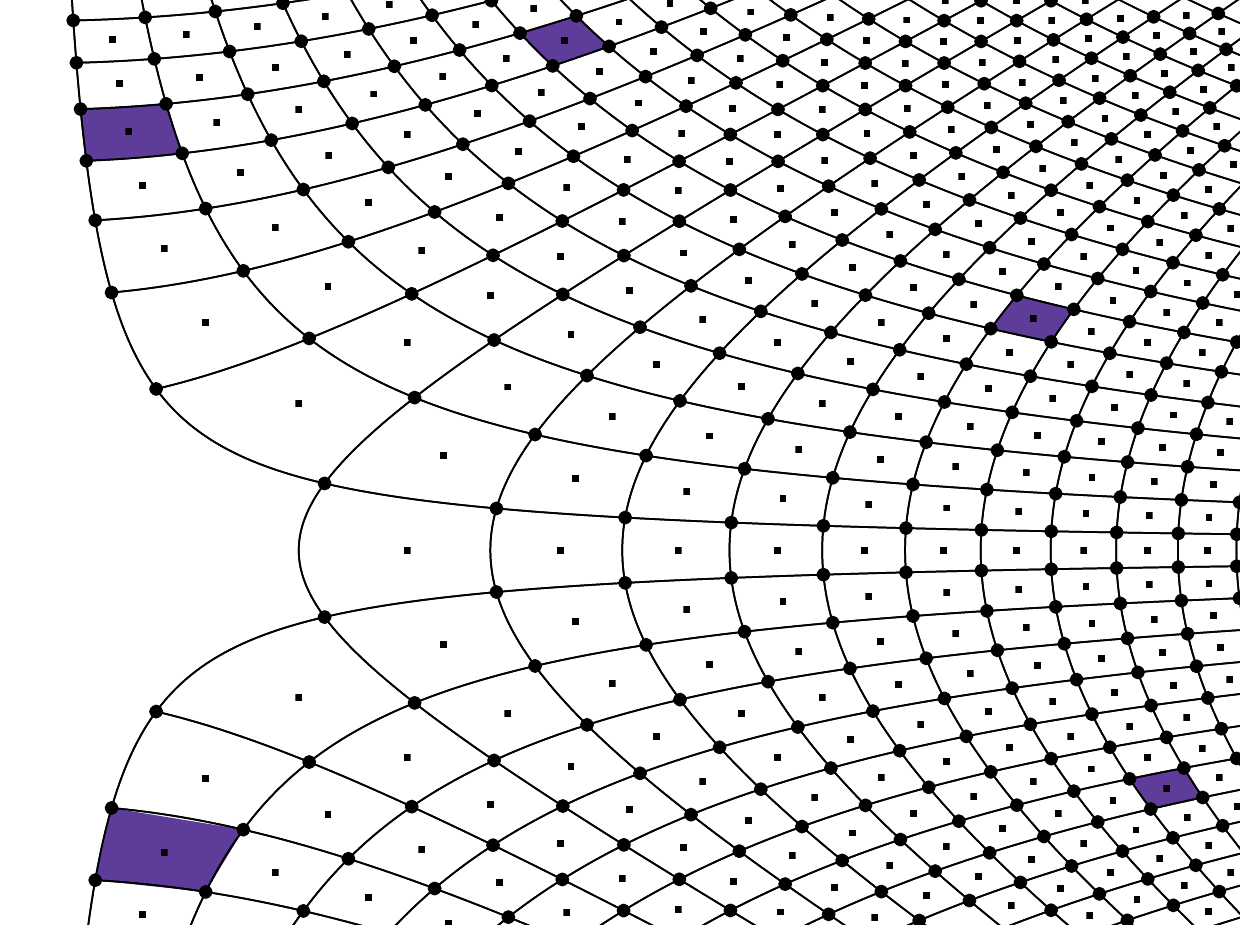}
    \caption{$\,$}\label{fig::intro2b}
  \end{subfigure}
  \captionsetup{width=0.95\linewidth}
  \caption{\corr A cracked crystalline material $\Lambda$ with $5$ possible dislocations cores highlighted in purple. In \subref{fig::intro2a} the original representation is used, whereas in \subref{fig::intro2b} the cracked lattice is mapped onto a distorted half-space lattice via the complex square root mapping. Note that while two dislocation cores near the crack surface are close to one another in the original description,  they are in fact significantly separated in the other description. }
\label{fig:intro2}
\end{figure}

\begin{mytheorem}[Theorem~\ref{thm:G}]\label{thm:B}
Let $\Gamma^* = \{(l_1,l_2) \in \Lambda^*\,\mid\, l_1 \leq 0,\; l_2 = 0\}$. There exists a Green's function $G\,:\,\La^*\times \La^* \to \R$ which satisfies 
\begin{subequations}\label{G-normal-eqns}
\begin{align}
-\Delta_{\rm d} G(l,s) &= \delta(l,s),\quad&&\text{ if } l \not\in\Gamma^*,\,s\not\in\Gamma^*,\label{G-normal-eqn1}\\
G(l,s) &= 0,\quad&&\text{ if } l \in \Gamma^*,\label{G-normal-eqn2}\\
G(l,s) &= G(s,l),\quad&&\text{ for all } l,s \in \La^*, \label{G-normal-eqn3}
\end{align}
\end{subequations}
where $\delta$ is the Kronecker delta and $-\Delta_{\rm d} G(l,s) = \sum_{\rho \in \Rc}\left(G(l,s) - G(l-\rho,s)\right)$ is the usual definition of a discrete Laplacian, with
\begin{equation}\label{Rc}
\Rc = \{ e_1,-e_1,e_2,-e_2\}.
\end{equation} 

In particular, for $l,s \in \Lambda^*$ and $\rho \in \Rc$, it satisfies
\[
|G(l+\rho,s) - G(l,s)| \leq  \frac{C}{(1+(1+|\omega(l)|)(|\omega(l)-\omega(s)|))},
\]
for some $C>0$, where $\omega$ is the complex square root mapping (defined explicitly in Section~\ref{sec:latmancom}). Finally, there exists $\epsilon > 0$ such that, for any $s \in \Lambda^*$, $\max_{l \in \La^*,\rho \in \Rc}|G(l+\rho,s)-G(l,s)| < \frac{1}{2} - \epsilon$.
\end{mytheorem}

The equilibrium configuration $y_{\rm s}$ from Theorem~\ref{thm:A} treats the crack surface $\Gamma$ merely as a geometric constraint in the form of a boundary and no actual crack opening is present in $y_{\rm s}$. The question of extending this construction to the case when both $m$ dislocations and a crack opening are present is addressed in the final result (c.f. Section \ref{sec:existence_Knot0}), which is achieved courtesy of the Implicit Function Theorem.
\begin{mytheorem}[Theorem~\ref{thm:Knot0}]\label{thm:C}
There exists a locally stable equilibrium configuration \linebreak ${y\,\colon\,\Lambda \to \R}$ which contains $m$ dislocations with the separation property as in Theorem~\ref{thm:A}, as well as a crack opening, in the sense that 
\[
\lim_{k \to \infty} |y(x_k^+) - y(x_k^-)| = \infty,
\]
where $x_k^{\pm} = (-k-\frac{1}{2},\pm\frac{1}{2})$. In particular, the dominant behaviour of this difference at infinity is characterised by a term $2K\sqrt{|x^+_k|}$ where $K > 0$ is small enough. This constant is effectively a measure of 'the strength' of the opening and is referred to in literature as the stress intensity factor.
\end{mytheorem}
}

\subsection{{ \corr New concepts}}\label{sec:int-ideas}
{\corr The crack surface breaks the notion of duality between $\Lambda$ and $\Lambda^*$ near $\Gamma$, which implies that the framework of the \emph{lattice complex} advanced in \cite{AO05,H17} cannot be readily applied. The key new concept of the present work is the introduction of a more general framework, which is termed as the \emph{lattice manifold complex}.

To preserve the notion of duality, the cracked crystal $\La$ is first mapped onto a distorted half-space using the complex square root mapping (see Figure~\ref{fig::intro2b}), followed by the creation of a second sheet of lattice by reflection over the half-space boundary. The complex square root mapping can then be inverted, giving rise to a manifold-like structure.

This construction leads to a delicate interplay between the notions of orientation and duality (both crucial in the algebraic-topology-based framework of CW complexes \cite{whitehead1949}) and the new notion of crack reflection symmetry. The key contribution of the paper is thus the introduction of this framework, both in terms of geometry and the resulting functional analytic setup. This is done in Section~\ref{sec:geo_func_frame}.

In particular, the work in this regard could be seen as a first example of a general new approach for handling the boundaries in finite and in particular non-convex discrete domains, centred around the distortion of the lattice via a conformal mapping that straightens the boundary. 

The concept of a near-crack-tip plasticity atomistic equilibrium, as introduced in \linebreak Theorem~\ref{thm:C}, opens the door to a number of further studies on this topic. The careful characterisation of the key components of such an equilibrium paves the way for a subsequent upscaling analysis and for studies of plasticity-induced fracture resistance. A thorough discussion about future research directions is presented in Section \ref{sec:conclusions}.
}
\subsection{Outline of the paper}
In Section~\ref{sec:geo_func_frame} the  geometric and functional framework of a lattice manifold complex is presented. Section~\ref{sec:latmancom} is devoted to introducing the lattice manifold complex, while Section~\ref{sec:pforms} discusses the resulting functional framework. The crucial notion of duality is described in Section~\ref{sec:dual_lmc}. Then, in Section~\ref{sec:dis_conf}, the notion of a dislocation configuration in a cracked crystal is introduced. Some of the more technical definitions related to Section~\ref{sec:geo_func_frame} are deferred to Appendix~\ref{app1}.

Section~\ref{sec:atom_model} describes the variational framework employed to study atomistic configurations. In Section~ \ref{sec:energy} the energy difference is introduced and equilibria-related definitions are stated. Section~\ref{sec:cont_pred} addresses the notion of the far-field continuum prediction in the framework of the lattice manifold complex. 

The main results are presented in Section~\ref{sec:main}. Section~\ref{sec:G} is devoted to the study of a Green's function on the dual lattice manifold complex, which will be shown to be by construction equivalent to considering a lattice Green's function in a cracked crystal with zero Dirichlet boundary condition. Such Green's function is proven to exist and its decay properties are characterised. This paves the way for the results concerning existence of dislocations-only equilibrium configurations, which are stated in Section~\ref{sec:existence_k0} and centre around an explicit construction relying on duality and the aforementioned Green's function. This is followed by Section~\ref{sec:existence_Knot0} in which existence of near-crack-tip plasticity equilibrium configurations is established, courtesy of preceding results and  an appropriate application of the Implicit Function Theorem.

The concluding remarks containing the outlook for future work are presented in Section \ref{sec:conclusions}, which is followed by proofs of the main results gathered in Section \ref{sec:proofs}.
\section{Geometric and functional framework}\label{sec:geo_func_frame}
In this section a geometric and functional framework well-suited to describing dislocations interacting with each other in the vicinity of a crack tip will be presented. The approach and notational conventions are based on \cite{H17}, in which the case of interacting dislocations without a crack present is considered, and also draws from and formalises the idea of a discrete complex square root manifold introduced in \cite{2018-antiplanecrack}. 

\subsection{Lattice manifold complex}\label{sec:latmancom}
The notational framework discussed in detail in  \cite[Section 2.1]{H17}, which in itself is based on the pioneering work in \cite{AO05}, can be extended to include the treatment of a {\it lattice manifold complex}, defined as a CW complex through the following construction. The reader is referred to Figure \ref{fig:lattice_manifold} for a visual intuition.

The underlying space of the complex is the Riemann surface $\Ss$ associated with the complex square root, defined as
\begin{equation}\label{S}
\mathcal{S}:= \{(z,w) \in \C^2\,|\, w^2 = z\},
\end{equation}
which is a connected {\corr Hausdorff} space \cite{napier_riemann_surfaces}. A  description of the manifold $\mathcal{S}$ to be used throughout is based on identifying $\C \cong \R^2$ and provided by the complex square root mapping $\omega\,:\,\R^2\setminus \{0\} \to \R^2$ defined as
\begin{equation}\label{om-map}
\omega(z) = \sqrt{r_z}\left(\cos(\theta_z/2),\sin(\theta_z/2)\right),
\end{equation}
where the polar coordinates are used with $z = r_z(\cos\theta_z,\sin\theta_z)$, where $r >0$ and $\theta \in (-\pi,\pi]$ (note the inclusion of the right end of the interval). It is easy to verify that if $(z,w) \in \mathcal{S}$ and $z\neq 0$, then, either $w \cong \omega(z)$ or $w \cong -\omega(z)$. From now on when there is no risk of confusion, the equality sign '$=$' shall be used in the place of '$\cong$' when discussing equivalent objects in $\R^2$ and $\C$.

\begin{figure}[!htbp]
  \begin{subfigure}[t]{.48\textwidth}
    \centering
    \includegraphics[width=\linewidth]{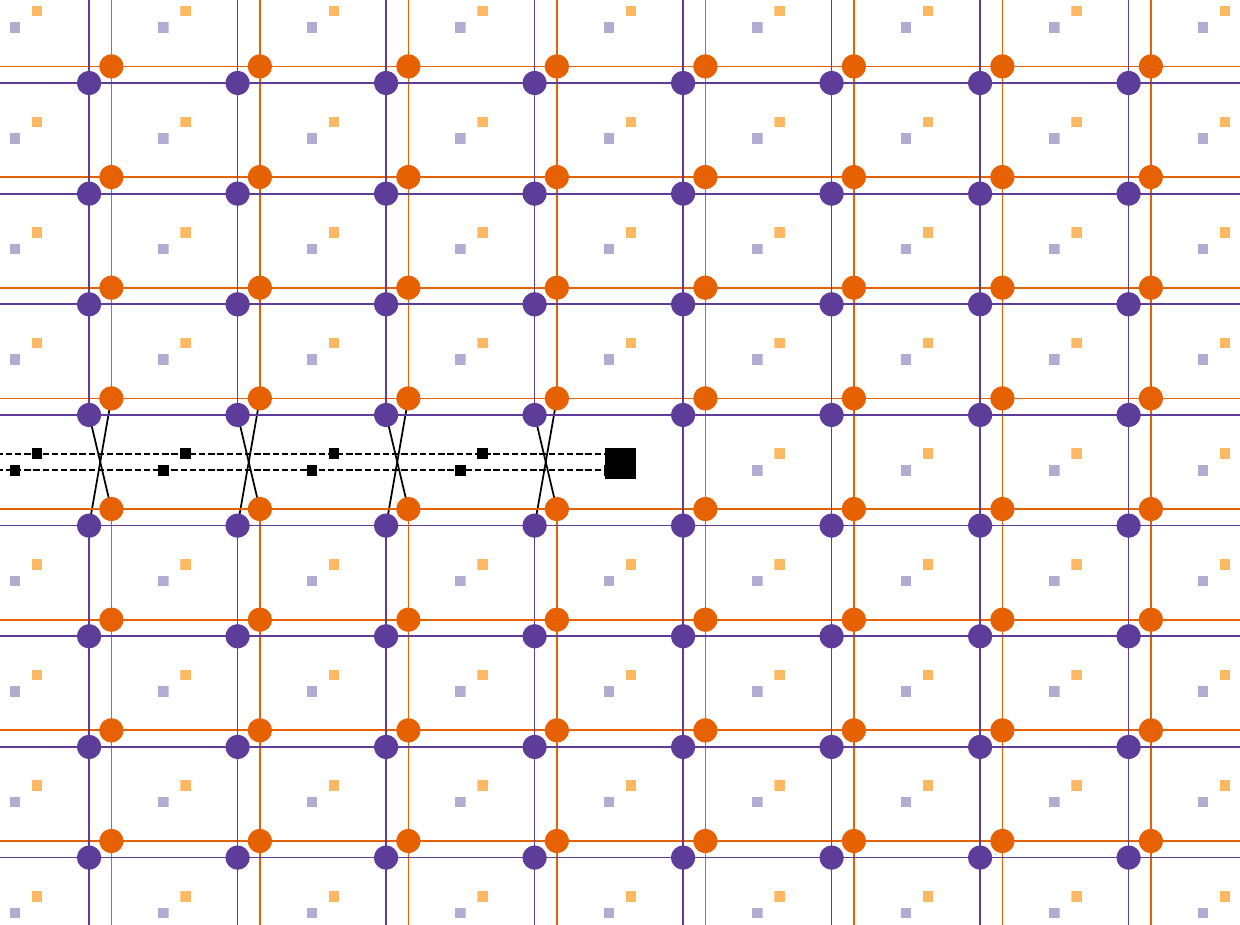}
    \caption{$\,$}\label{fig::lat_z}
  \end{subfigure}
	\quad
  \begin{subfigure}[t]{.48\textwidth}
    \centering
    \includegraphics[width=\linewidth]{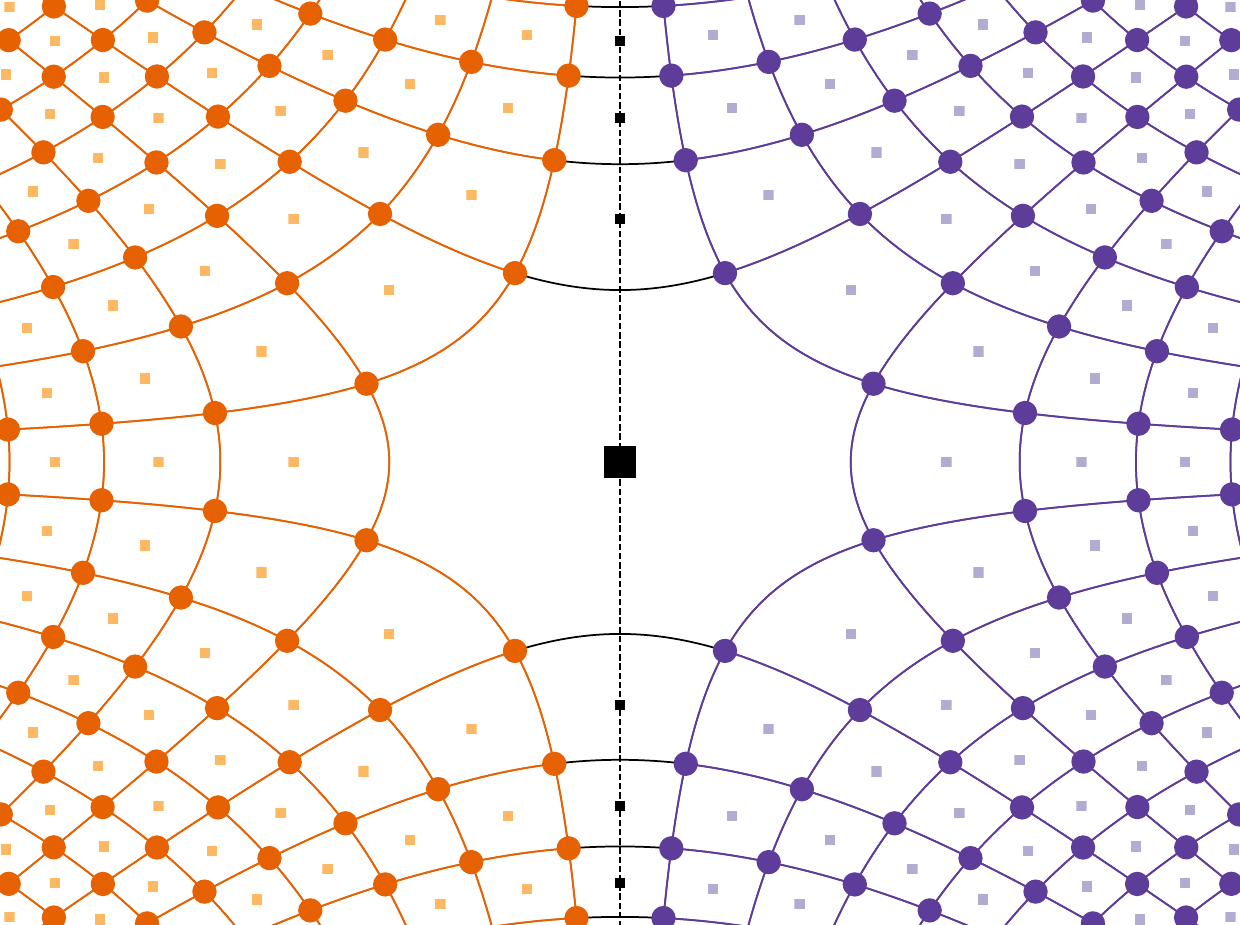}
    \caption{$\,$}\label{fig::lat_w}
  \end{subfigure}
  \captionsetup{width=0.95\linewidth}
  \caption{The geometry of manifolds $\Mc$ and $\Wc$, which together describe the complex square root manifold $\Ss$. {\corr In \subref{fig::lat_z} the $\Mc$-manifold description is depicted}, with $\Mc_0$ consisting of two coinciding lattices $\Mc^{+}_0$ (purple dots) and $\Mc^-_0$ (orange dots) and lines denoting elements of $\Mc_1$. The large black square represents the crack tip position, taken to be the centre of the coordinate system. The crack $\Gamma^{\Mc}_0$, through which the lattices are glued is denoted by a double dashed line. The little squares depict the $0$-cells of the dual lattice manifold $\Mc_0^*$. {\corr In \subref{fig::lat_w} the $\Wc$-manifold description is shown}, with the distorted double lattice $\Wc_0$ consisting of purple dots, corresponding to $\omega_{\Mc}\left(\Mc_0^+\right)$ and orange dots, corresponding to $\omega_{\Mc}\left(\Mc^-_0\right)$ and lines corresponding to elements of $\Wc_1$. The crack $\Gamma^{\Wc}_0$ is shown as the vertical dashed black line $\{w_1 = 0\}$, with the large black square representing the crack tip. The little  squares depict the $0$-cells of the dual distorted lattice $\Wc_0^*$.}
\label{fig:lattice_manifold}
\end{figure}

The fact that $w = \pm \omega(z)$ highlights the general construction: two copies of $\R^2 =  \C$ are taken and 'glued' together at the branch cut
\begin{equation}\label{Gamma0}
 \Gamma_0 := \{(x_1,0)\,|\,x_1 \leq 0\},
\end{equation}
which {\corr is the continuum (of infinitesimal thickness) counterpart to the atomistic crack surface $\Gamma$ introduced in Section~\ref{sec:out-res} and Figure~\ref{fig:intro1}}. This motivates defining 
\begin{equation}\label{Mc}
\Mc := \left(\R^2  \times \{+1\}\right)\cup\left(\left(\R^2\setminus\{(0,0)\}\right)  \times \{-1\}\right),
\end{equation}
which in the light of the above discussion is isomorphic to
\begin{equation}\label{Ssw}
\Wc:= \{w \in \R^2\,|\, \exists z \in \R^2\,\text{ such that } (z,w) \in \Ss\} = \R^2
\end{equation}
through the mapping $\omega_{\Mc}\,:\,\Mc \to \Wc$ applied to $k = (k_x,k_b) \in \Mc$, where $k_x \in \R^2$ and $k_b \in \{-1,1\}$ and given by
\begin{equation}\label{om_mc}
\omega_{\Mc}(k) = \begin{cases}
+\omega(k_x) \quad&\text{if } k_x \neq (0,0),\,k_b = +1,\\
-\omega(k_x) \quad&\text{if } k_x \neq (0,0),\,k_b = -1,\\
\; (0,0)\quad\quad&\text{if } k_x = (0,0),\,k_b = +1.
\end{cases}
\end{equation}
Likewise, the mapping $k \mapsto (k_x,\omega_{\Mc}(k))$ is an isomorphism between $\Mc$ and $\Ss$, thus emphasising that all three descriptions are equivalent.

For future reference it is also useful to introduce 
\begin{equation}\label{Mcpm}
\Mc^{\pm}:= \{(k_x,k_b) \in \Mc\,|\, k_b = \pm 1\}
\end{equation}
and 
\begin{equation}\label{Wcpm}
\Wc^{\pm}:= \omega_{\Mc}(\Mc^{\pm}).
\end{equation}

The $\Mc$-manifold equivalent of $\Gamma_0$ defined in \eqref{Gamma0} is given by $\Gamma_0^{\Mc}:=\Gamma_0^+ \cup \Gamma_0^-$, where
\begin{equation}\label{McGamma0}
\Gamma_0^{+}:= \left(\Gamma_0 \times \{+1\}\right),\quad \Gamma_0^-:= \left(\left(\Gamma_0 \setminus \{(0,0)\}\right)\times \{-1\}\right),
\end{equation}
which also gives rise to the vertical line
\begin{equation}\label{WcGamma0}
\omega_{\Mc}(\Gamma_0^+) \cup \omega_{\Mc}(\Gamma_0^-) = \{ w = (w_1,w_2) \in \Wc\;|\; w_1 = 0\} =: \Gamma_0^{\Wc}.
\end{equation}
Both $\Gamma_0^{\Mc}$ and $\Gamma_0^{\Wc}$ are visualised in Figure \ref{fig:lattice_manifold}.

With the underlying spaces of $\Ss$ and $\Mc$ introduced, the square lattice manifold can now be defined as
\begin{equation}\label{LSs}
\Ss_0:= \left\{(z,w) \in \mathcal{S}\,|\, z - \left(\frac{1}{2},\frac{1}{2}\right)\in \Z^2\right\}.
\end{equation}
An equivalent description is given by 
\begin{equation}\label{LMc}
\Mc_0 := \left(\Z^2 - \left(\frac{1}{2},\frac{1}{2}\right)\right)\times\{-1,1\},
\end{equation}
which is depicted in Figure \ref{fig::lat_z}. The mapping $\omega_{\Mc}$ from \eqref{om_mc} provides an isomorphism from $\Mc_0$ to 
\begin{equation}\label{Lssw}
\Wc_0:=\{w \in \R^2\,|\,\exists z \in \R^2\,\text{ such that }(z,w) \in \Ss_0\},
\end{equation}
which is a deformed lattice depicted in Figure \ref{fig::lat_w}. It again makes sense to distinguish
\begin{equation}\label{LMcpm}
\Mc^{\pm}_0:=\{(k_x,k_b) \in \La_{\Mc}\,|\,k_b = \pm 1\}.
\end{equation}

While the $\Ss$-manifold description introduced in \eqref{S} and \eqref{LSs} provides a simple definition underpinning the essence of the approach, in practice the $\Mc$-manifold description introduced in \eqref{Mc}, \eqref{Mcpm}, \eqref{LMc}, \eqref{McGamma0}, \eqref{LMcpm} and the $\Wc$-manifold description introduced in \eqref{Ssw}, \eqref{WcGamma0} and \eqref{Lssw} will be primarily employed. In particular, deformations of the cracked crystal will be represented by deformations of $\Mc_0$, whereas the lattice manifold complex structure is more easily discussed in the $\Wc$-description and can be introduced as follows.

In the abstract framework of CW complexes, as introduced by Whitehead in \cite{whitehead1949}, $\Wc$ is considered the underlying space, whereas the set of $0$-cells is based upon $\Wc_0$. The remaining sets of $1$-cells and $2$-cells are defined iteratively: for $p\geq 1$, a $p$-dimensional cell is $e \subset \Wc$ for which there exists a homeomorphism mapping the interior of the $p$-dimensional closed ball in $\R^p$ onto $e$, and mapping boundary of the ball onto a finite union of cells of dimension less than $p$. This has to be done in such a way that $\Wc = \R^2$ is the disjoint union of its $0$-, $1$- and $2$-cells.

The set of all $p$-cells is denoted by $\Wc_p$, whereas, mildly abusing the notation, the entire complex is referred to by $\Wc$ itself.

The isomorphism $\omega_{\Mc}$ defined in \eqref{om_mc}  maps $\Wc_p$ to $\Mc_p$, a set of $p$-cells of the complex $\Mc$ (again mildly abusing the notation). 

Away from the vertical line $\Gamma_0^{\Wc}$ in \eqref{WcGamma0} representing the crack surface, the construction of $1$- and $2$-cells of $\Wc$ is as in \cite{H17}, since one can map $\Wc_0$ back to $\Mc_0$ defined in \eqref{LMc}, which can then be treated as two separate lattices $\Mc_0^+$ and $\Mc_0^-$. 

The $\Wc$-description then permits defining the remaining sets of $1$-cells as shown in Figure \ref{fig::lat_w}, which automatically ensures a complete description of $2$-cells too through the iterative procedure described in the previous paragraph. Note that in particular the border of the central $2$-cell $E_0$ consists of eight $1$-cells.

\begin{figure}[!htbp]
  \begin{subfigure}[t]{.48\textwidth}
    \centering
    \includegraphics[width=\linewidth]{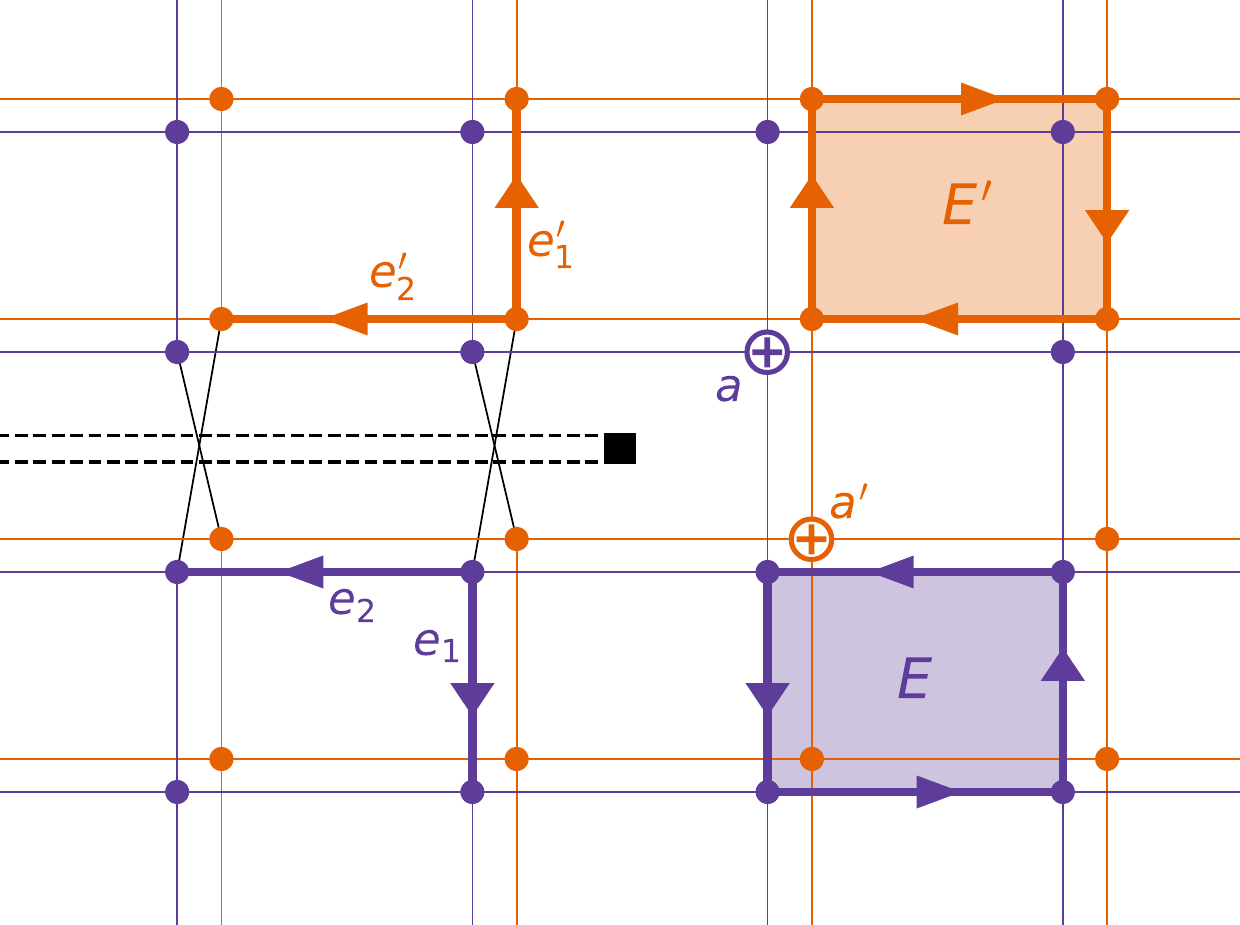}
    \caption{$\,$}\label{fig::1}
  \end{subfigure}
	\quad
  \begin{subfigure}[t]{.48\textwidth}
    \centering
    \includegraphics[width=\linewidth]{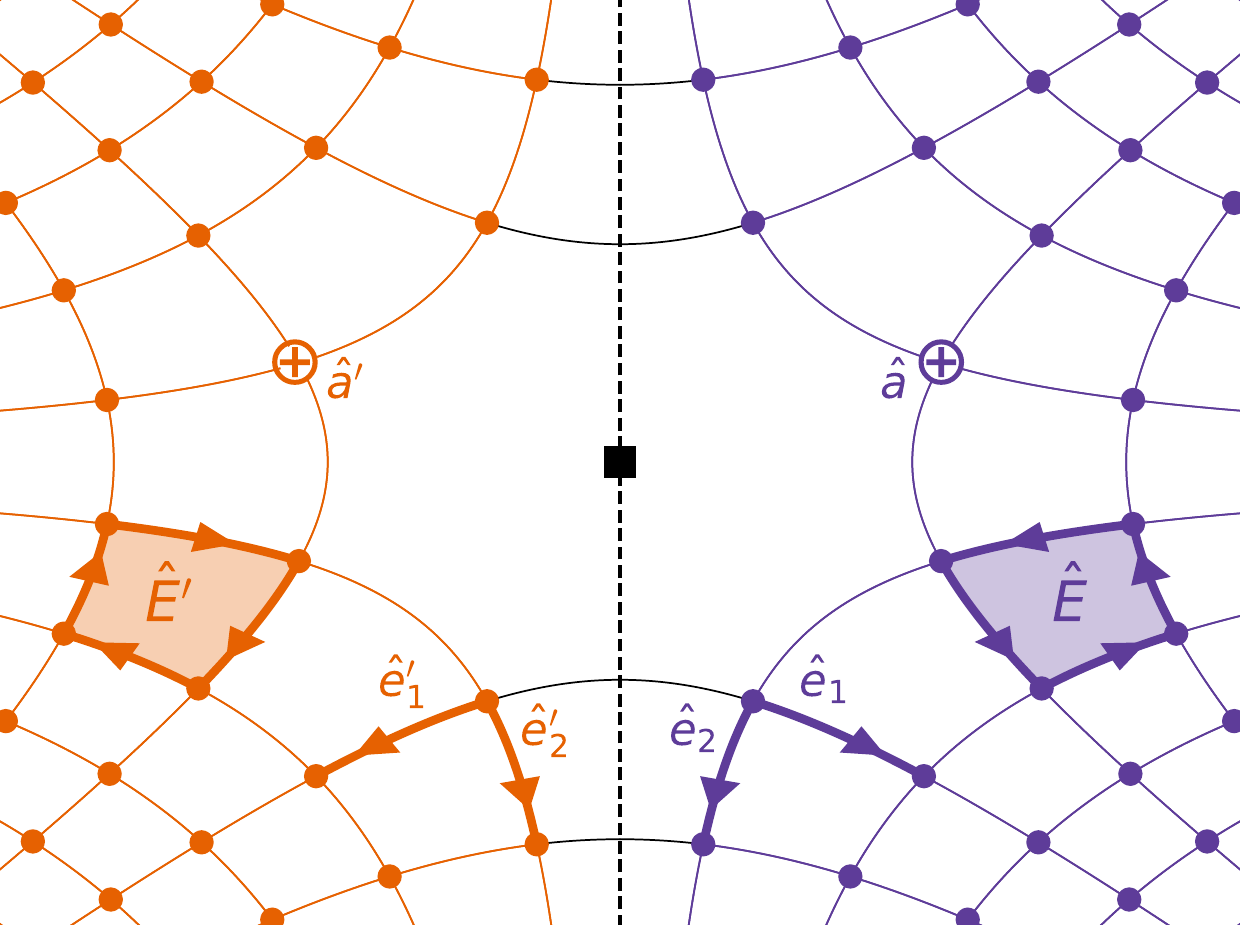}
    \caption{$\,$}\label{fig::2}
  \end{subfigure}
  \captionsetup{width=0.95\linewidth}
  \caption{The lattice manifold complex, with the notion of orientation and crack reflection symmetry highlighted. {\corr In \subref{fig::1} the $\Mc$-complex is depicted}, with $p$-cells in $\Mc^+$ in purple and $p$-cells in $\Mc^-$ in orange, together with highlighted cells: a positively oriented $0$-cell $a \in \Mc_0^+$ and its crack reflection ${a'\in \Mc_0^-}$ (with orientation preserved), two $1$-cells $e_1,e_2 \in \Mc_1^+$ and their crack reflections $e_1',e_2' \in \Mc_1^-$, a negatively (anti-clockwise) oriented $2$-cell $E \in \Mc_2^+$ and its crack reflection $E' \in \Mc_2^-$, which is positively (clockwise) oriented (note the change in orientation). {\corr In \subref{fig::2} the equivalent $\Wc$-description is shown, which is } obtained through $\omega_{\Mc}$ mapping. Corresponding $p$-cells and their crack reflections in the $\Wc$-complex with notation $\hat{e} \equiv \omega_{\Mc}(e)$. }
\label{fig:3}
\end{figure}

The CW complex machinery further enriches the description by prescribing each cell $e \in \Wc_p$ with orientation,  with $-e$ used to denote the same cell (as a subset of the underlying space) but with reversed orientation. As a result, the boundary operator $\partial$ and the co-boundary operators $\delta$ can be defined, with $\partial$ mapping an oriented $p$-cell to its consistently oriented boundary, which is a collection of $(p-1)$-cells, and $\delta$ mapping an oriented $p$-cell to a collection of $(p+1)$-cells containing it in its boundary.

To distinguish when a cell is treated just as a subset of the underlying space, for $e \in \Wc_p$, the notation $\chi(e)$ is used to denote the space it occupies in $\Wc = \R^2$, which implies $\chi (e) = \chi (-e)$. If $e \in \Mc_p$ then $\chi(e)$ is defined as $\chi(e):= \omega_{\Mc}^{-1} \left(\chi(\omega_{\Mc}(e))\right)$. In particular, it is worth distinguishing
\begin{equation}\label{Mcp_pm}
\Mc_p^{\pm} := \{ e \in \Mc_p\;|\; \chi(e) \subset \Mc^{\pm}\},
\end{equation}
and
\begin{equation}\label{Wcp_pm}
\Wc_p^{\pm} := \omega_{\Mc}^{-1}(\Mc_p^{\pm}).
\end{equation}

The notion of orientation is intertwined with the crucial notions of duality in a general CW complex,and of {\it crack reflection symmetry} in the lattice manifold complex, with the latter intuitively a consequence of the line symmetry across $\Gamma_0^{\Wc}$ in the $\Wc$-description and the resulting one-to-one correspondence between {\it lattices }$\Mc_0^+$ and $\Mc_0^-$. The notational convention is that a crack reflection of a $p$-cell $e$ is denoted by $e'$ and the construction in particular is such that $\left(\Mc_p^{\pm}\right)' = \Mc_p^{\mp}$ and hence $\left(\Wc_p^{\pm}\right)' = \Wc_p^{\mp}$.

With the underlying structure spatially inhomogeneous and manifold-like in nature, the interplay between orientation, duality and crack reflection is quite nuanced and thus worth spelling out in full detail. This is, in the interest of clarity of exposition, deferred to Appendix \ref{app1} and the reader is referred to Figure~\ref{fig:3}, Figure~\ref{fig:9} and Figure~\ref{fig:6} for an intuitive visual description.
\begin{remark}\label{not-abuse1}
To keep the notation concise, it is convenient to commit several natural abuses of notation, which will be listed here for clarity. Firstly, the notation $\Wc$ interchangeably refers to both {\it the underlying space} defined in \eqref{Ssw} and the resulting CW complex, which consists of both the space and the collection of cells $\bigcup_p \Wc_p$. Secondly, the notation $\Wc_0$ can refer to both {\it the lattice manifold} defined in \eqref{Lssw}, and {\it the space }of $0$-cells of the complex $\Wc$, with the difference being that a $0$-cell is additionally assigned an orientation. Same notational conventions apply to equivalent objects in the $\Mc$-manifold description.
\end{remark}
\subsection{The spaces of $p$-forms preserving crack symmetries}\label{sec:pforms}
In order to discuss deformations of a cracked crystal, a consistent way of defining functions on 

A general definition of the space of $p$-forms defined on the set of $p$-cells $\Omega_p$ of a CW complex $\Omega$ to be used throughout is given by
\begin{equation}\label{Vsc}
\Vsc(\Omega_p):= \{f\,:\,\Omega_p\to \R \;|\;f(e) = -f(-e), \text{ for any } e \in \Omega_p\}.
\end{equation}
From now on either $\Omega = \Wc$ or $\Omega = \Mc$. 
\begin{remark}\label{not-abuse2}
Complementary to Remark \ref{not-abuse1}, it will often be convenient to define functions in $\Vsc(\Omega_p)$ simply through specifying the values they take on positively oriented $p$-cells only, which in the light of \eqref{Vsc} is enough to uniquely determine them.
\end{remark}

To reflect the fact that the crystalline system considered is cracked, the presence of the crack has to be encoded in the functional setup. In \cite{2018-antiplanecrack} this is done by considering the square lattice $\Mc_0^+$ defined in \eqref{LMcpm} and explicitly removing interactions across $\Gamma_0$ defined in \eqref{Gamma0}. In the follow-up study in \cite{2019-antiplanecrack} this is instead encoded in the interatomic potential employed. 

In the present study the former approach is more natural and in particular in the framework of lattice manifold complexes this can be achieved with the help of crack reflection symmetry described visually in Figure~\ref{fig:3}, with full details presented in Appendix \ref{app1}.

The space of $p$-forms preserving crack symmetries is defined as
\begin{equation}\label{p-form-space-s}
\Vsc_{\rm s}(\Omega_p):= \Big\{f \in \Vsc(\Omega_p)\;|\; f(e') = f(e) \Big\},
\end{equation}
which is a vector space under pointwise addition. 

\begin{remark}\label{rem-crack-sym}
As a consequence of the crack reflection symmetry, if $f \in \Vsc_{\rm s}(\Wc_p)$, then 
\begin{equation}\label{Wc_w1}
\forall\, e \in \Wc_p \text{ with } \chi(e) \cap \Gamma_0^{\Wc} \neq \emptyset,\; f(e) = 0,
\end{equation}
which implies that $f$ is uniquely determined by the values it takes on $\Wc_p^+$, defined in \eqref{Wcp_pm}. This is because $(\Wc_p^+)' = \Wc_p^-$ and any $e \not\in (\Wc^+_p)\cup(\Wc_p^-)$ is such that $e \cap \Gamma_0^{\Wc} \neq \emptyset$.

By the same logic, if $g \in \Vsc_{\rm s}(\Mc_p)$, then it is uniquely determined by the values it takes on $\Mc^+_p$, defined in \eqref{Mcp_pm}, with $g(e) = 0$ for any $e \not\in (\Mc^+_p)\cup(\Mc_p^-)$, defined in \eqref{Wcp_pm}. This underlines that the setup employed is equivalent to considering a single cracked crystal lattice. 
\end{remark}

Since the underlying assumption is that the crystal lattice is of infinite size, it is convenient to define the space of crack-symmetric $p$-forms with compact support,
\begin{equation}\label{p-form-comp}
\Vsc_{\rm c}(\Omega_p):= \{f \in \Vsc_{\rm s}(\Wc_p)\,|\, \supp(f)\,\text{ is compact in } \Omega \}
\end{equation}
and also introduce, for any $f \in \Vsc_{\rm s}(\Omega_p)$,
\begin{equation}\label{inf-norm}
\|f\|_{\infty}:= \sup_{e \in \Omega_p}|f(e)|.
\end{equation}
For any $f \in \Vsc(\Omega_p)$ and any finite $A \subset \Omega_p$, one can define the integral
\begin{equation}\label{def-int}
\int_A f := \sum_{e \in A} f(e).
\end{equation}
The differential operator $\pmb{d}\,:\,\Vsc_{\rm s}(\Omega_p) \to \Vsc_{\rm s}(\Omega_{p+1})$ and the co-differential operator $\pmb{\delta}\,:\,\Vsc_{\rm s}(\Omega_p) \to \Vsc_{\rm s}(\Omega_{p-1})$ are given by
\begin{equation}\label{d-delta}
\pmb{d}f(e) := \int_{\partial e} f = \sum_{\tilde{e} \in \partial e} f(\tilde{e}),\quad \pmb{\delta}f(e) := \int_{\delta e} f = \sum_{\tilde{e} \in \delta e} f(\tilde{e}). 
\end{equation}
where the boundary operator $\partial$ and the co-boundary operator $\delta$ are discussed in Section~\ref{sec:latmancom}. The fact that both $\pmb{d}f$ and $\pmb{\delta}f$ preserve the crack reflection symmetry follows naturally from the definition of $\Vsc_{\rm s}(\Omega_p)$ in \eqref{p-form-space-s} and how reflections are defined in \eqref{wc0-refl} and \eqref{wc12-refl}.

The notion of integration introduced in \eqref{def-int} permits defining the bilinear form 
\begin{equation}\label{def-inner}
(f,g) := \int_{\Omega_p}f\,g,
\end{equation}
which is well-defined when at least one of $f$ or $g$ lies in the space $\Vsc_{\rm c}(\Omega_p)$ and leads to an integration by parts formula
\begin{equation}\label{def-int-parts}
(\pmb{d}f,g) = (f,\pmb{\delta}g),
\end{equation}
holding true for $f \in \Vsc_{\rm c}(\Omega_p)$ and $g \in \Vsc_{\rm c}(\Omega_{p+1})$.

The bilinear form in \eqref{def-inner} motivates defining
\begin{equation}\label{def-L2}
\mathscr{L}^2(\Omega_p):= \{f \in \Vsc_{\rm s}(\Omega_p)\,|\, (f,f) < +\infty\},
\end{equation}
which can be shown to be a Hilbert space with a norm $\|\cdot\|:= (\cdot,\cdot)^{1/2}$ induced by the inner product $(\cdot,\cdot)$. Of particular interest in the subsequent analysis is the discrete Sobolev space 
\begin{equation}\label{Hcc}
\Hcc(\Omega_0):=\{u \in \Vscs(\Omega_0)\;|\; \|\pmb{d}u\|_{\mathscr{L}^2(\Omega_1)} < \infty\,\text{ and }\, u(e_0) = 0\},
\end{equation}
where, if $\Omega = \Mc$, then $e_0$ is such that $\chi(e_0) =\left(\tfrac{1}{2},\tfrac{1}{2},+1\right)$ and, if $\Omega = \Wc$, then $e_0$ is such that $\chi(e_0)=\omega_{\Mc}\left(\tfrac{1}{2},\tfrac{1}{2},+1)\right)$. The choice of $e_0$ is arbitrary and ensures that only one constant displacement lies in the space.

The discrete Laplace operator $\pmb{\Delta}\,:\,\Vsc(\Omega_0) \to \Vsc(\Omega_0)$ is also introduced, given by \begin{equation}\label{def-hlaplace}
\pmb{\Delta}f := \pmb{\delta d}f,
\end{equation}
as it will be needed to provide a characterisation of equilibria in Section \ref{sec:existence_k0}.

The isomorphism $\omega_{\Mc}$ from \eqref{om_mc} translates to the spaces of $p$-forms, with 
\[
f \in \Vsc(\Wc_p) \iff \left(f \circ \omega_{\Mc}\right) \in \Vsc(\Mc_p),
\]
with similar statements for spaces $\Vsc_{\rm s}(\Wc_p), \Vsc_{\rm c}(\Wc_p), \mathscr{L}^2(\Wc_p)$ and $\dot{\mathscr{H}}^1(\Wc_0)$. 



\subsection{Dual lattice manifold complex}\label{sec:dual_lmc}
The construction of the lattice manifold complex preserves the crucial notion of duality, permitting defining the corresponding {\it dual lattice manifold complex} through the usual construction, as discussed in \cite[Section 2.2]{H17}. The visual guide is provided by Figure \ref{fig:9}.

\begin{figure}[!htbp]
  \begin{subfigure}[t]{.48\textwidth}
    \centering
    \includegraphics[width=\linewidth]{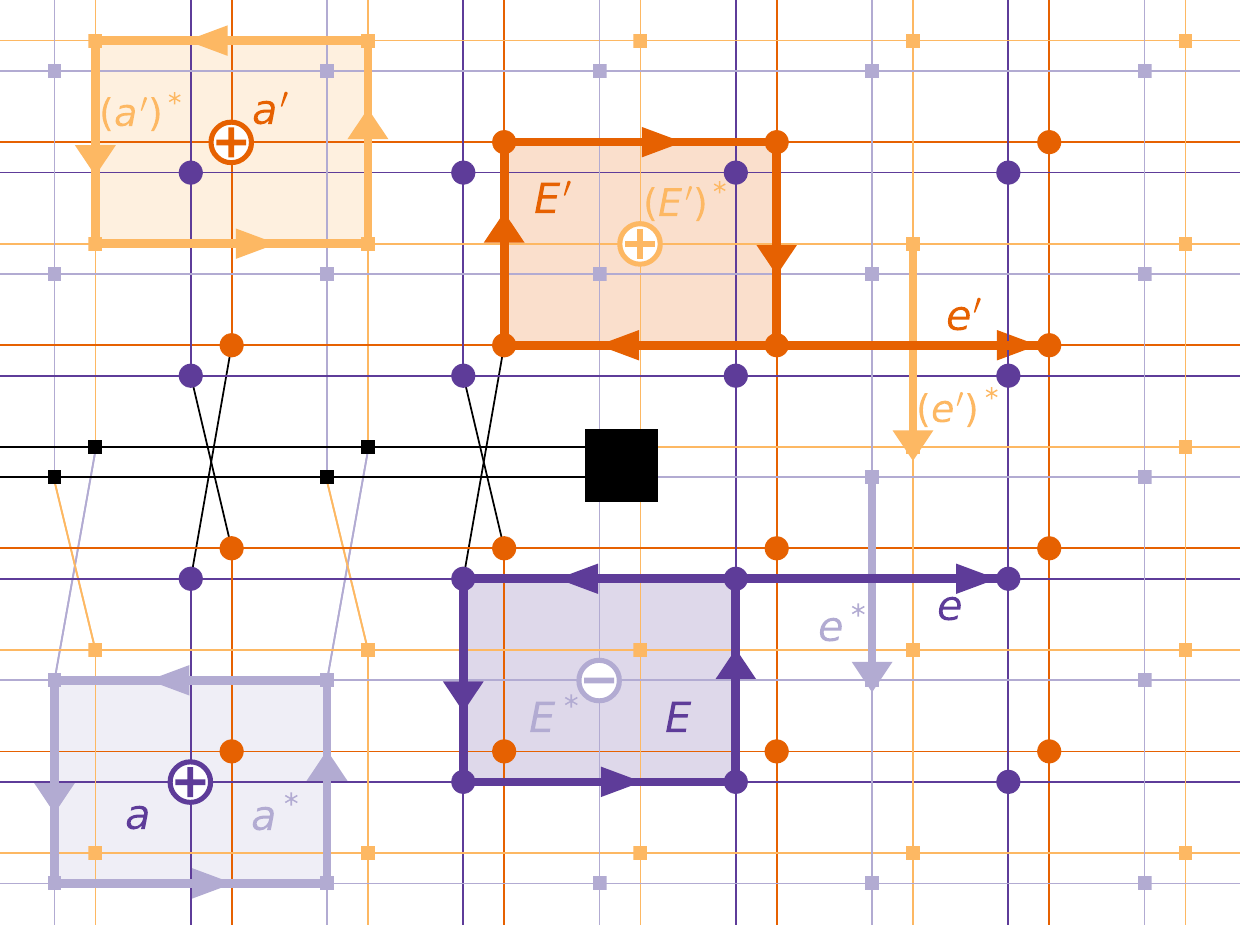}
    \caption{$\,$}\label{fig::7}
  \end{subfigure}
	\quad
  \begin{subfigure}[t]{.48\textwidth}
    \centering
    \includegraphics[width=\linewidth]{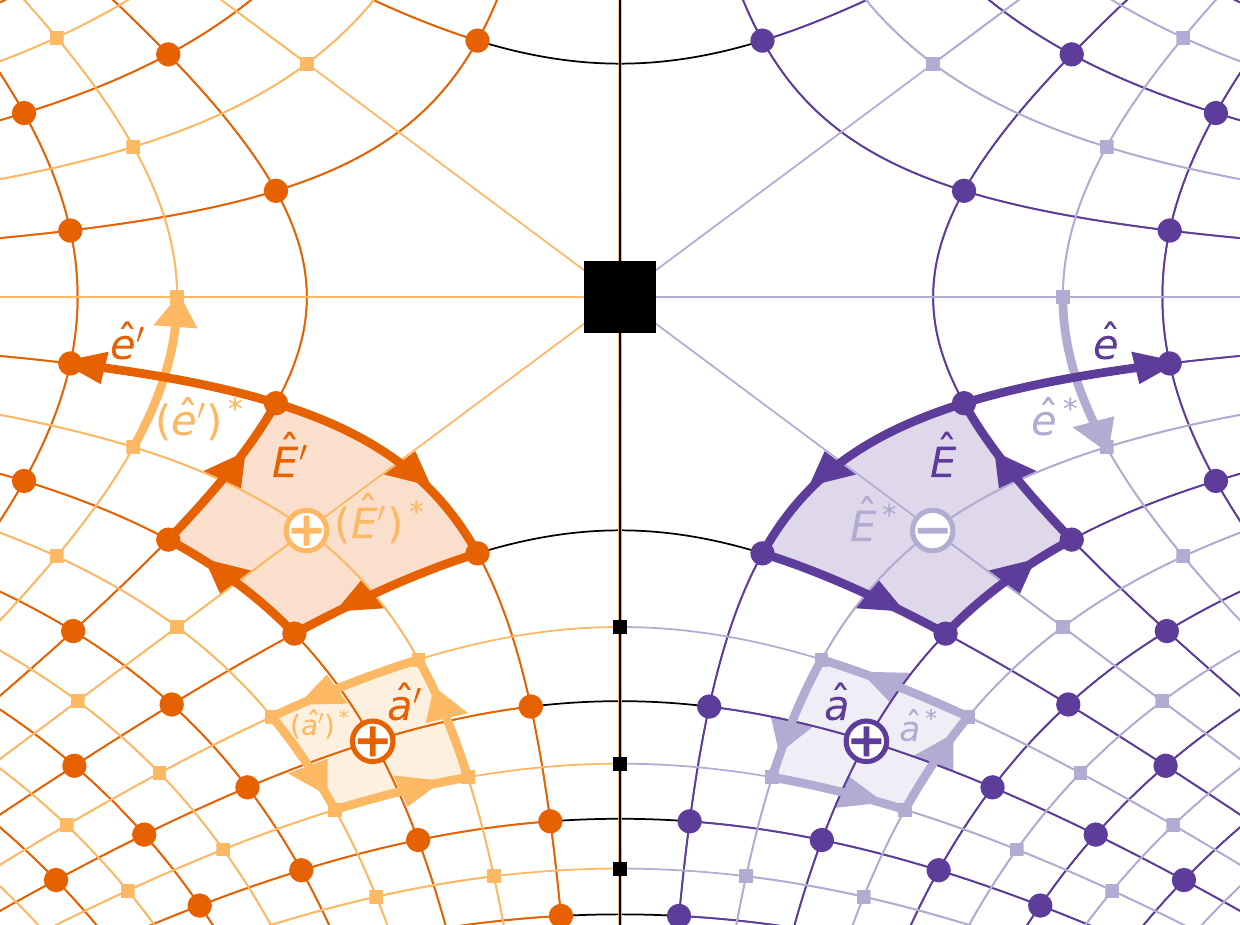}
    \caption{$\,$}\label{fig::8}
  \end{subfigure}
  \captionsetup{width=0.95\linewidth}
  \caption{The duality in the framework of lattice manifold complex, with the differing notions of orientation and crack reflection symmetry emphasised. {\corr In \subref{fig::7}} the primal complex $\Mc$ is depicted in darker colours with dual complex $\Mc^*$ in lighter, with several examples of $p$-cells $e \in \Mc_p$ and their duals $e^* \in \Mc^*_{2-p}$ as well as their crack reflections $e'$ and $(e')^*${\corr : a positively oriented $0$-cell $a \in \Mc_0^+$ and its dual $a^* \in \Mc^{*,+}_2$ together with their crack reflections ${a'\in \Mc_0^-}$ and $(a')^* \in \Mc_2^{-,*}$; a $1$-cell $e \in \Mc_1^+$ and its dual $e^* \in \Mc_1^{*,+}$, together with their crack reflections $e' \in \Mc_1^-$ and $(e')^* \in \Mc_1^{*,-}$; a negatively oriented $2$-cell $E \in \Mc_2^+$ and its dual $E^* \in \Mc_0^{*,+}$ together with their crack reflections $E' \in \Mc_2^-$ and $(E')^* \in \Mc_0^{*,-}$.} {\corr In \subref{fig::8} the equivalent $\Wc$-description is shown, which is} obtained through $\omega_{\Mc}$ mapping. Corresponding $p$-cells and their duals and crack reflections are highlighted{\corr , with notation $\hat{e} \equiv \omega_{\Mc}(e)$.}}
\label{fig:9}
\end{figure}

The set of $0$-cells of the dual lattice manifold complex $\Ss^*$ is given by
\begin{equation}\label{DLSs}
\Ss_0^*:= \left\{(z,w) \in \mathcal{S}\,|\, z \in \Z^2\right\}
\end{equation}
and likewise, for the dual $\Mc^*$, one has
\begin{equation}\label{DLMc}
\Mc_0^* :=  \Mc_0^{*,+} \cup \Mc_0^{*,-},
\end{equation}
where 
\begin{equation}\label{DLMc2}
\Mc_0^{*,+}:= \left(\Z^2\times\{1\}\right)\quad \text{ and }\quad \Mc_0^{*,-}:=  \left(\left(\Z^2\setminus\{(0,0)\}\right) \times \{-1\}\right)
\end{equation}
and similarly, for $\Wc^*$,  
\[
\Wc_0^*:= \{w \in \R^2\,|\, \exists z \in \R^2\,\text{ such that } (z,w) \in \Ss_0^*\}.
\]
Figure \ref{fig:lattice_manifold} provides a visual representation of $\Mc_0^*$ and $\Wc_0^*$.  

The set of $p$-cells of the dual lattice manifold complex $\Wc^*$ is denoted by $\Wc_p^*$, which, through the isomorphism $\omega_{\Mc}$ in \eqref{om_mc} gives rise to $\Mc_p^*$. The duality between complexes $\Wc$ and $\Wc^*$ is expressed through the isomorphism 
\begin{equation}\label{iso}
^*\,:\,\Wc_p \to \Wc^*_{2-p},\text{ for }p=0,1,2.
\end{equation}

The notion of orientation on the dual lattice manifold complex and the action of the isomorphism $^*$ is visually presented in Figure~\ref{fig:9}, with the full discussion deferred to Appendix \ref{app1}. The underlying idea behind the construction is to ensure that the boundary operator $\partial^*$ and the co-boundary operator $\delta^*$ on the dual lattice manifold complex satisfy 
\begin{equation}\label{def-bond-op-dual}
\partial^*e^* = (\delta e)^*\,\text{ and }\,\delta^*e^* = (\partial e)^*,
\end{equation}
which will prove crucial when describing atomistic equilibria in Section \ref{sec:atom_model}.   

The notion of crack reflection symmetry introduced visually for the primal complex in Figure \ref{fig:3} can be extended to the dual lattice manifold by defining for any $e^* \in \Wc_p^*$ its reflection as
\begin{equation}\label{wc*-refl}
(e^*)' := (e')^*.
\end{equation}
The resulting notion of crack reflection symmetry on the dual lattice manifold complex is visually presented in Figure \ref{fig:6}, with full details presented in Appendix \ref{app1}. It is in particular stressed that, as shown in Figure \ref{fig::5}, the resulting notion of symmetry on the dual lattice manifold complex no longer coincides with a simple line symmetry  across $\Gamma_0^{\Wc}$.

\begin{figure}[!htbp]
  \begin{subfigure}[t]{.48\textwidth}
    \centering
    \includegraphics[width=\linewidth]{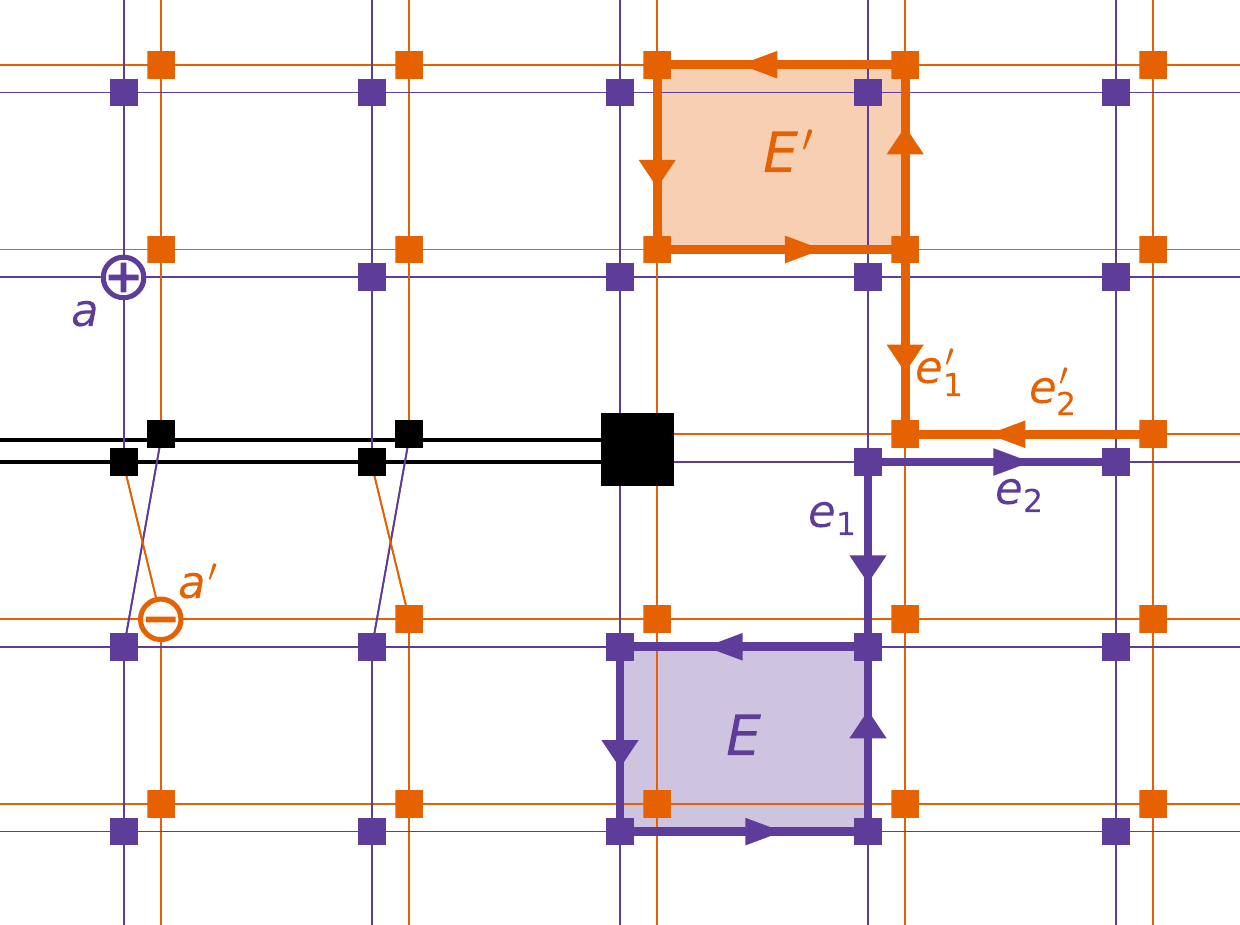}
    \caption{$\,$}\label{fig::4}
  \end{subfigure}
	\quad
  \begin{subfigure}[t]{.48\textwidth}
    \centering
    \includegraphics[width=\linewidth]{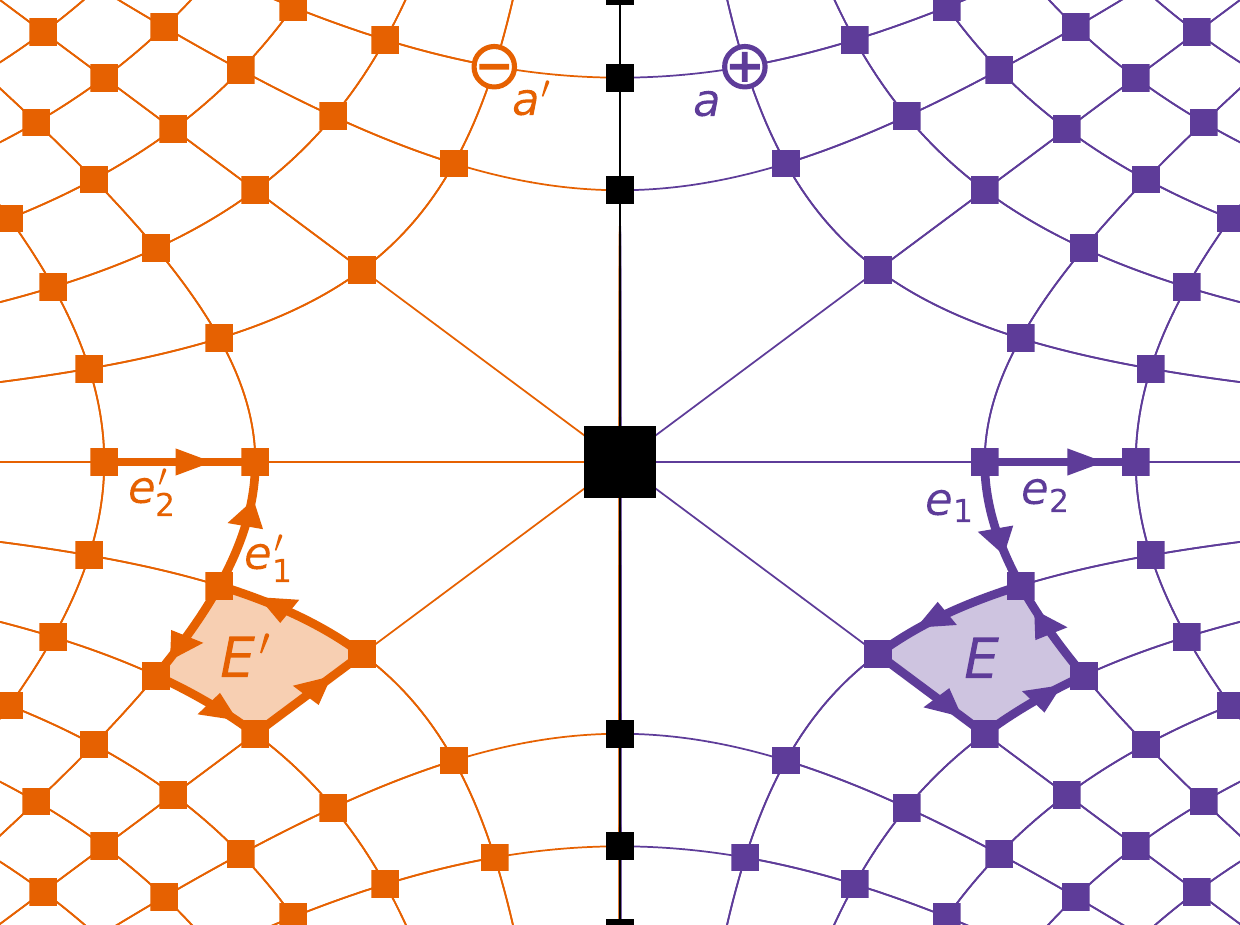}
    \caption{$\,$}\label{fig::5}
  \end{subfigure}
  \captionsetup{width=0.95\linewidth}
  \caption{The dual lattice manifold complex, with the notion of orientation and crack reflection symmetry highlighted. {\corr In  \subref{fig::4} the $\Mc^*$ complex is depicted }with $p$-cells in $\Mc^{*,+}$ in purple and $p$-cells in $\Mc^{*,-}$ in orange, together with highlighted cells: a positively oriented $0$-cell $a \in \Mc_0^{*,+}$ and its crack reflection $a'\in \Mc_0^{*,-}$ (with orientation swapped), two $1$-cells $e_1,e_2 \in \Mc_1^{*,+}$ and their crack reflections $e_1',e_2' \in \Mc_1^{*,-}$, a positively (anti-clockwise) oriented $2$-cell $E \in {\corr \Mc_2^{*,+}}$ and its crack reflection $E' \in {\corr \Mc_2^{*,-}}$, with orientation preserved. {\corr In  \subref{fig::2} the equivalent $\Wc^*$-description is shown, which is }obtained through $\omega_{\Mc}$ mapping. Corresponding $p$-cells and their crack reflections in the $\Wc^*$-complex with notation $\hat{e} \equiv \omega_{\Mc}(e)$.}
\label{fig:6}
\end{figure}

As on the original complex, for any $e^* \in \Mc_p^*$ there exists a unique $\omega_{\Mc}^{-1}(e^*) \in \Wc_p^*$, and thus the isomorphism $^*\,:\,\Mc_p \to \Mc_{p-2}^*$ is defined, for $e \in \Mc_p$ by
\[
e^*:= \omega_{\Mc}^{-1}\left(\omega_{\Mc}(e)^*\right),
\]
which entails that the discussion above applies to $\Mc^*$, with the crack reflection symmetry of $e^* \in \Mc_p$ being given by $(e^*)' = \omega_{\Mc}^{-1}\left(\omega_{\Mc}((e^*)')\right)$. The relevant visual depiction is given in Figure \ref{fig::4}.

The useful one-to-one correspondence discussed in \eqref{Mcp_pm} and \eqref{Wcp_pm} translates to the dual lattice manifold complex, thus if 
\begin{equation}\label{Mcp_pm_dual}
\Mc_p^{*,\pm} := \{ e \in \Mc^*_p\;|\; \chi(e) \subset \Mc^{\pm}\},
\end{equation}
where $\Mc^{\pm}$ was defined in \eqref{Mcpm}, then $(\Mc_p^{*,+})' = \Mc_p^{*,-}$, and likewise, if 
\begin{equation}\label{Wcp_pm_dual}
\Wc_p^{*,\pm} := \omega_{\Mc}^{-1}(\Mc_p^{*,\pm}),
\end{equation}
then $(\Wc_p^{*,+})' = \Wc_p^{*,-}$. 

Let $\Omega$ represent either the complex $\Wc$ or the complex $\Mc$. Any crack symmetry preserving function $f \in \Vsc_{\rm s}(\Omega_p)$ gives rise to a $f^*\,:\,\Omega_{p-2}^* \to \R$ through
\[
f^*(e^*) := f(e)
\]
and thanks to the way duality and reflection are defined, it is ensured that
\[
f^*((e^*)') = f^*((e')^*) = f(e') = f(e) = f^*(e^*),
\]
which implies that $f^* \in \Vsc_{\rm s}(\Omega_{p-2}^*)$. 

In fact by construction $^*\,:\, \Vsc_{\rm s}(\Omega_p) \to \Vsc_{\rm s}(\Omega_{2-p}^*)$ is an isomorphism and similarly $^*$ is also an isomorphic mapping between the Hilbert spaces $\mathscr{L}^2(\Omega_p)$ and $\mathscr{L}^2(\Omega_{2-p}^*)$.

It is further noted that {\corr the notion of} duality extends to the differential operator \linebreak ${\pmb{d}^*\,:\, \Vsc_{\rm s}(\Omega^*_p) \to \Vsc_{\rm s}(\Omega^*_{p+1})}$ and the co-differential operator $\pmb{\delta}^*\,:\,\Vsc_{\rm s}(\Omega^*_{p} \to \Vsc_{\rm s}(\Omega^*_{p-1})$, which are defined as 
\[
\pmb{d}^*f^*(e^*) := \int_{\partial^* e^*} f^* = \int_{\delta e} f = \pmb{\delta}f(e)\,\text{ and }\,\pmb{\delta}^*f^*(e^*)\int_{\delta^* e^*} f^* = \int_{\partial e} f = \pmb{d}f(e).
\]
\begin{remark}\label{rem:dual-dirchlet}
In the light of Remark \ref{rem-crack-sym}, for any $f^* \in \Vsc_{\rm s}(\Omega_p^*)$, as
a consequence of \eqref{Wc_w1},
\[
\forall e^* \in \Omega_p^*\;\text{ with } \chi(e^*) \cap \Gamma_0^{\Omega} \neq\emptyset,\;f^*(e^*) = 0. 
\]
Note that on the dual complex if $\chi(e^*) \cap \Gamma_0^{\Omega} \neq \emptyset$ then in fact $\chi(e^*) \subset \Gamma^{\Omega}_0$, thus what on the original lattice manifold complex is a {\it zero Neumann boundary restriction} arising from the crack, on the dual complex takes the form of a {\it zero Dirichlet boundary restriction}.
\end{remark}
The discrete Sobolev space on the dual lattice manifold respecting the boundary condition due to crack reflection symmetry is given by
\begin{equation}\label{Hcc-dual}
\Hcc(\Omega_0^*):= \{ u \in \Vsc_{\rm s}(\Omega_0^*)\;|\; \|\pmb{d}u\|_{\mathscr{L}^2(\Omega_1^*)} < \infty \}.
\end{equation}
There is no need for an additional requirement in relation to admissible  constant functions, since by definition any $u \in \Vsc_{\rm s}(\Omega_0^*)$ satisfies $u(e) = 0$ for any $e$ such that $\chi(e) \subset \Gamma_0^{\Omega}$. 

Finally, it is noted that the convenient notational abuses discussed in Remark \ref{not-abuse1} also apply to $\Wc^*$, $\Mc^*$ as well as $\Wc^*_0$ and $\Mc^*_0$. 
\subsection{Dislocation configurations in a cracked crystal}\label{sec:dis_conf}
In this section the concepts introduced in \cite{2014-dislift} about permissible dislocation configurations will be adapted to take an existing crack into account.

For any $y \in \Vscs(\Mc_0)$, the set of bond-length $1$-forms is defined as
\begin{equation}\label{bond-length}
[\pmb{d}y]:= \{\alpha \in \Vscs(\Mc_1)\,|\,\|\alpha\|_{\infty} \leq \tfrac{1}{2}\,\text{ and }\, \forall e \in \Mc_1,\,\alpha(e) - \pmb{d}y(e) \in \Z\}.
\end{equation}
This definition is motivated by the fact that any two displacements $y,\tilde{y} \in \Vscs(\Mc_0)$ such that, for all $e \in \Mc_0$, $y(e) - \tilde{y}(e) \in (\Z+C)$ for some constant $C \in \R$ represent a physically equivalent three-dimensional configuration which is reflected by $[\pmb{d}y] = [\pmb{d}\tilde{y}]$. 

A dislocation core is any positively oriented $2$-cell $e \in \Mc^+_2$ for which $\pmb{d}\alpha(e) = \int_{\partial e}\alpha \neq 0$. If $\mu \in \Vscs(\Mc_2)$ is such that its image is contained in $\{-1,0,1\}$, then $y \in \Vscs(\Mc_0)$ is regarded as a displacement containing dislocation configuration $\mu$ if 
\[
\exists\,\alpha \in [\pmb{d}y]\,\text{ such that }\, \pmb{d}\alpha  = \mu.
\]
Note that $\mu \in \Vscs(\Mc_2)$ represents the Burgers vectors \cite{hirth-lothe} associated with each $2$-cell.

{\corr For the subsequent results about existence of equilibria, the dislocations are required to satisfy a \emph{minimum separation property}, as first noted in Section~\ref{sec:out-res}.} A possible way of defining the set of admissible dislocation configurations {\corr satisfying this property} is as follows. Fix $\epsilon > 0$ and some $\pmb{b}=(b_i)_{i=1}^m$ with $b_i \in \{\pm 1\}$ representing the Burgers vector {\corr (in the present context determining the 'charge' of the dislocation, as discussed at length in \cite{hirth-lothe})}  and define the set of $2$-forms
\begin{align}
\mathscr{B}_n:= \Big\{&\mu = \sum_{i=1}^m b_i(\mathbb{1}_{e_i} + \mathbb{1}_{e_i'})\;|\; e_i \in \Mc^+_2,\label{Bn}\\
& {\corr \max\left\{D(e_i,e_j),D_{\Wc}(e_i,e_j)\right\} }\geq \epsilon n,\, \text{ for all } i,j \in \{1,\dots,m\},i\neq j  \Big\}.\nonumber
\end{align}
Here $\mathbb{1}_{e_i}$ denotes a $2$-form indicator function, which can more generally be defined as a $p$-form given by 
\begin{equation}\label{indicator_f}
\mathbb{1}_{e}(\tilde{e}):=\begin{cases} \pm 1&\quad \tilde{e} = \pm e,\\
0&\quad \text{otherwise}.
\end{cases}
\end{equation}
The fact that $e_i'$ appears in \eqref{Bn} is due to the crack reflection symmetry, since if $\mu \in \Vsc_{\rm s}(\Mc_2)$, then if $\mu(e) = \pm 1$, i.e. there is a dislocation core at a positively oriented $e$, then by construction there is a further dislocation core at a positively oriented $-e'$  with $\mu(-e') = \mp 1$. 

{\corr A rigorous definition of the distance functions $D(\cdot,\cdot)$ and  $D_{\Wc}(\cdot,\cdot)$} in \eqref{Bn} will be presented in the Appendix (Section~\ref{sec:distance}), as part of a broader discussion about measuring distances in the lattice manifold complex framework. {\corr In particular, Section \ref{sec:min-sep-prop} will be devoted to showing why the minimum separation property is posed as a maximum over two notions of a distance. Intuitively, for $e_i,e_j \subset \Mc^+$, $D(e_i,e_j)$ coincides with the Euclidean distance on $\R^2$, whereas $D_{\Wc}(e_i,e_j)$ is the Euclidean distance on $\R^2$ between $\omega_{\Mc}(e_i) \in \Wc_2^+$ and $\omega_{\Mc}(e_j) \in \Wc_2^+$.} As $n$ in the definition of $\mathscr{B}_n$ in \eqref{Bn} increases, so does the separation of each defect thus many of the subsequent results will hold for $n$ large enough. The reason why $\epsilon >0$ is fixed is to prepare the ground for a future study of upscaling \cite{BvM}, in which the lattice manifold will be rescaled by $\frac{1}{n}$, thus ensuring that in the limit two defects do not collapse onto one another.

A domain with interactions across the crack disregarded, as in the case of $\Mc$, does not, by itself, ensure that every resulting configuration has a crack opening along $\Gamma_0$ present. A set of admissible displacements containing a crack can be defined as 
\begin{equation}\label{Vcrack}
\mathscr{V}^{\rm crack}:= \{ y \in \Vscs(\Mc_0)\,|\,y\left(e_m^+\right) - y\left(e_m^-\right) \to \infty\,\text{ as }m \to \infty \},
\end{equation}
where
\[
e_m^{\pm} := \left(\left(-m - \frac{1}{2}, \pm \frac{1}{2}\right),1,1\right) \in \Mc_0
\]
A displacement containing a crack opening and a given permissible dislocation configuration is thus any $y \in \mathscr{V}^{\rm crack}$ for which there exists $\alpha \in [\pmb{d}y]$ such that $\pmb{d}\alpha \in \mathscr{B}_n$.
\section{The atomistic model}\label{sec:atom_model}
With the geometric and functional setup introduced, the atomistic model can now be discussed. Firstly, the notion of an atomistic energy difference will be presented, followed by a discussion about an appropriate far--field predictor which will ensure that the energy is well-defined over an appropriate function space. This will be followed by an explicit construction of dislocations-only equilibria in a cracked domain, which exploits the notion of duality for the lattice manifold complex. This construction will pave the way for the subsequent result about the existence of near-crack-tip plasticity equilibria. 
\subsection{Energy difference}\label{sec:energy}
With the underlying lattice manifold complex $\Mc$ introduced in Section \ref{sec:latmancom}, for a pair of displacements $y, \tilde{y} \in \Vscs(\Mc_0)$, the energy difference is defined to be
\begin{equation}\label{energy}
\E(y,\tilde{y}) := \int_{\Mc_1}\left(\psi(\pmb{d}y) - \psi(\pmb{d}\tilde{y})\right),
\end{equation}
where
\begin{equation}\label{psi}
\psi(r) = \frac{\lambda}{2}\dist(x,\Z)^2
\end{equation}
is a $1$-periodic quadratic pair potential with $\lambda > 0$. The $1$-periodicity of the potential reflects the fact that the underlying crystal is three-dimensional with each lattice site $e \in \Mc_0$ representing a column of atoms that are $1$-periodic in the direction perpendicular to the plane considered. This is further discussed near the definition of $[\pmb{d}y]$ in \eqref{bond-length}.

The crack reflection symmetry together with the fact that ${\psi(-r) = \psi(r)}$ ensures that 
\[
\int_{\Mc_1} \psi(\pmb{d}y) = 2 \int_{\Mc_1^+} \psi(\pmb{d}y),
\]
thus implying that the energy difference considered here is equivalent to the one considered in \cite{2013-disl,EOS2016, H17,2018-antiplanecrack,2019-antiplanecrack}.

It is clear that the energy difference in \eqref{energy} is well-defined if $y-\tilde{y} \in \Vsc_{\rm c}(\Mc_0)$ and a straightforward adjustment of results in \cite{2013-disl} ensures that, under certain conditions to be subsequently discussed, $\E$ can be extended by continuity {\corr to} displacements such as $y -\hat{y} \in \Hcc(\Mc_0)$ defined in \eqref{Hcc}.

{\corr Throughout the paper the primary interest lies in identifying (locally) stable equilibrium configurations, which can be defined in the following way.}
\begin{definition}\label{def:stability}
A displacement $y \in \Vsc_{\rm s}(\Mc_0)$ is a locally stable equilibrium if there exists $\epsilon > 0$ such that $\E(y + \tilde{y}, y) \geq 0$ for all $\tilde{y} \in \Vsc_{\rm c}(\Mc_0)$, defined in \eqref{p-form-comp}, with $\|\pmb{d} \tilde{y}\|_{\mathscr{L}^2(\Mc_0)} \leq \epsilon$.
\end{definition}
It follows from \cite[Lemma 3.3]{2013-disl} that under the potential given by $\eqref{psi}$, any locally stable equilibrium $y \in \Vsc_{\rm s}(\Mc_0)$ satisfies $\pmb{d}y(e) \not\in (\Z + \frac{1}{2})$ for all $e \in \Mc_1$ and thus any locally stable equilibrium $y$ further satisfies what is referred to in literature as {\it discrete ellipticity} or {\it strong stability}, namely 
\begin{equation}\label{strong-s}
\int_{\Mc_1} \psi''(\pmb{d}y)\,(\pmb{d}v)^2 \geq \lambda \|\pmb{d}v\|^2_{\mathscr{L}^2(\Mc_1)},\quad \forall v \in \Vsc_{\rm c}(\Mc_0).
\end{equation}
{\corr This property will prove useful in Section \ref{sec:existence_Knot0}, as it will ensure that the Implicit Function Theorem \cite{serge} is applicable.}  

As discussed in \cite[Section 3.1]{H17}, the periodicity of the potential $\psi$ implies that any single locally stable equilibrium gives rise to an entire family of equilibria, {\corr each differing only by the displacement of at least one atom jumping by an integer, or by a shift of displacements of all atoms by a constant.} This can be captured by defining an equivalence relation
\begin{equation}\label{equiv-rel}
y \sim \tilde{y} \iff\,\forall e \in \Mc_0,\quad y(e) - \tilde{y}(e) \in \left(\Z + C\right)\,\text{ for some } C \in \R.
\end{equation} 

{\corr The strategy employed to find locally stable equilibrium $y$ containing both a crack opening and dislocations concerns setting 
\begin{equation}\label{hatu_tildeu}
y = \hat{u} + u\quad\text{ and }\quad \tilde{y} = \hat{u},
\end{equation}
where $\hat{u}$ is a far-field predictor which should accurately capture the behaviour of the equilibrium configuration away from the crack defect core at the origin. On the other hand $u$ is constrained to lie in the space $\Hcc(\Mc_0)$ and is interpreted as an atomistic corrector.} 

The issue of deriving an appropriate $\hat{u}$ will be addressed now.

\subsection{Far-field prediction}\label{sec:cont_pred}
{\corr With the underlying domain of infinite size, the atomistic energy difference defined in \eqref{energy} can be used to model a crack opening provided that the predictor-corrector approach championed in \cite{2013-disl}, \cite{EOS2016} and \cite{2018-antiplanecrack} is employed.}

{\corr As noted in \cite{2018-antiplanecrack}}, in the case of a single Mode III anti-plane crack problem posed on $\R^2$ with the crack tip at the centre of the coordinate system and the crack surface given by $\Gamma_0$ in \eqref{Gamma0}, the predictor $\hat{u}_{\rm c}\,:\,\R^2\setminus \Gamma_0 \to \R$ is required to satisfy 
\begin{subequations}\label{PDE-c}
\begin{align}
-\Delta \hat{v}_{\rm c} = 0\,&\text{ in }\, \R^2\setminus \Gamma_0,\\
\nabla \hat{v}_{\rm c} \cdot \nu = 0\,&\text{ on }\,\Gamma_0\setminus\{(0,0)\}
\end{align}
\end{subequations}
and an explicit solution ensuring local integrability near the crack tip \cite{SJ12} is given in polar coordinates $x = (r\cos\theta,r\sin\theta)$ by
\begin{equation}\label{c-pred-simple}
K\hat{v}_{\rm c}(r,\theta):= K\sqrt{r}\sin\tfrac{\theta}{2},
\end{equation} 
which in fact coincides with the second component of the complex square root mapping $\omega$ from \eqref{om-map}. Here $K$ is the (rescaled) stress intensity factor and its sign determines the orientation of the crack, much like the sign of $b_i$ in \eqref{Bn} determines the charge of a dislocation. By convention and without loss of generality, it is assumed that $K \geq 0$. 

A lattice manifold equivalent of the canonical crack predictor $\hat{v}_{\rm c}$ can be defined as a $\hat{u}_{\rm c}\,:\,\Mc\setminus\{(0,0)\} \to \R$ given by
\begin{equation}\label{uhatcmc}
K\hat{u}_{\rm c}(k):=\begin{cases} &+K\hat{v}_{\rm c}(k_x)\quad\text{if }\, k_b = +1,\\\ &-K\hat{v}_{\rm c}(k_x)\quad\text{if }\, k_b = -1.
\end{cases}
\end{equation}
The crack reflection symmetry across the branches of the manifold ensures that, recalling Remark \ref{not-abuse2}, the definition in \eqref{uhatcmc} fully specifies a function  $\hat{u}_{\rm c} \in \Vscs(\Mc_0)$.

In the particular case of no screw dislocations present, corresponding to setting $m=0$ in the definition of admissible dislocation configurations in \eqref{Bn}, it has been established in \cite{2018-antiplanecrack}, for $\psi \in C^4(\R)$, that the mapping
\[
\Hcc(\Mc_0) \ni u \mapsto \E(K\hat{u}_{\rm c} + u,K\hat{u}_{\rm c}),
\]
which corresponds to setting $\hat{u}$ from \eqref{hatu_tildeu} to be $K\hat{u}_{\rm c}$, is well-defined and admits a locally unique minimiser $\bar{u}_{\rm c} \in \Hcc(\Mc_0)$ for the stress intensity factor $K$ in \eqref{c-pred-simple} small enough, thus giving rise to an equilibrium of the form
\begin{equation}\label{yc}
y_{\rm c}:= K\hat{u}_{\rm c} + \bar{u}_{\rm c},
\end{equation}
which is locally stable in the sense of Definition~\ref{def:stability}. As discussed in \cite[Remark 3.2]{2014-dislift} and also \cite[Lemma 3.3]{2013-disl}, global smoothness of $\psi$ is not in fact a necessary condition and it is straightforward to see that the same is true for $\psi$ defined in \eqref{psi}.

Interestingly, the fact that, for $K$ small enough, $\|\pmb{d}y_{\rm c}\|_{\infty} < \frac{1}{2}$, defined in \eqref{inf-norm}, is a direct consequence of linear dependence in $K$ of both $\hat{u}_{\rm c}$ and $\bar{u}_{\rm c}$ (\cite[Theorem 2.4]{2018-antiplanecrack}).

As was noted at the end of Section \ref{sec:dis_conf}, there is a crucial distinction to be made between modelling screw dislocations simply in a cracked lattice domain, that is with interactions across a crack surface disregarded, as is inherent in the lattice manifold complex setup, and modelling screw dislocations in the presence of an actual crack opening. The former case corresponds to setting the stress intensity factor $K$ in \eqref{c-pred-simple} to $K=0$, i.e. effectively setting $\hat{u}_{\rm c} \equiv 0$, and the latter to considering $K > 0$.
  
In what follows first the case $K=0$ is handled, which will be shown to permit an explicit construction of a locally stable atomistic equilibrium $y_{\mu} \in \Vsc_{\rm s}(\Mc_0)$ containing dislocation configuration $\mu \in \mathscr{B}_n$, thus circumventing the need to derive a corresponding far-field predictor. 

This can subsequently be used to assert existence of a locally stable {\it near-crack-tip plasticity} equilibrium in the case when $K > 0$, by setting $\hat{u} = K\hat{u}_{\rm c} + y_{\mu}$ in \eqref{hatu_tildeu}. 

\section{Main Results}\label{sec:main}
With the geometric and functional framework introduced and the atomistic model {\corr rigorously defined}, the main results of the paper will now be presented.

The first set of results concerns the notion of a Green's function on the dual lattice manifold complex, which is equivalent to a lattice Green's function in the crack geometry with zero Dirichlet boundary condition. The results proven concern existence of such a Green's function and a careful characterisation of its decay properties. 

This naturally leads to results establishing existence, uniqueness and an explicit construction of locally stable dislocation-only equilibrium configurations in a cracked domain (the case when the stress intensity factor $K=0$). This is achieved by exploiting the duality of the lattice manifold complex, which implies that the problem at hand is closely related to the notion of a Green's function on the dual lattice manifold.

Notably, highlighting the benefit of detailed treatment of the crack surface by means of the lattice manifold complex machinery, a usual restriction on the minimum separation distance between the dislocations and the boundary (the crack surface) is not needed.  

This lays foundation for the remaining results, which include asserting existence and uniqueness of locally stable {\it near-crack-tip plasticity equilibria} in which both a crack opening and a fixed number of dislocations is present, corresponding to the case when the stress intensity factor $K$ defined in \eqref{c-pred-simple} is strictly positive and the number of dislocations in the definition of $\mathscr{B}_n$ from \eqref{Bn} is $m > 0$. 

\subsection{Green's function on the dual lattice manifold complex}\label{sec:G}
In this section only the dual lattice manifold is considered, thus the $^*$-superscript is only kept when referring to the spaces of $p$-cells, but dropped when referring to individual $p$-cells, as well as $p$-forms and boundary, co-boundary, differential, and co-differential operators.

The following definition lays out the relevant notion of a Green's function.  
\begin{definition}\label{def:G}
A function $\G\,:\,\Mc_0^* \times \Mc_0^* \to \R$ is said to be a Green's function on the dual lattice manifold complex if it satisfies the following properties.
\begin{enumerate}
\item {\it Crack reflection symmetry:} for any $\tilde{e} \in \Mc^*_0$, 
\[
\G(\cdot,\tilde{e}) \in \Vsc_{\rm s}(\Mc_0^*).
\]
\item {\it Variable symmetry:} for any $e,\tilde{e} \in \Mc_0^*$, it holds that 
\[
\G(e,\tilde{e}) = \G(\tilde{e},e).
\]
\item {\it A fundamental solution:} if $\tilde{e} \subset \Mc^{\pm}$, then 
\[
\pmb{\Delta}\G(e,\tilde{e}) = \mathbb{1}_{\tilde{e}} + \mathbb{1}_{\tilde{e}'},
\]
where the Laplace operator defined in \eqref{def-hlaplace} is applied with respect to the first variable and the indicator function is defined in \eqref{indicator_f}.
\end{enumerate}
\end{definition}
It is noted that (3) is the defining property of any Green's function, with $\mathbb{1}_{\tilde{e}'}$ appearing due to condition (1), which is there to ensure that the said Green's function is relevant to the problem at hand. 

Furthermore, condition (1) by construction implies that if $e_0 \cap \Gamma_0^{\Mc} \neq \emptyset$ then $\G(e_0,\tilde{e}) = 0$, since, as explained in Section \ref{sec:dual_lmc}, in the dual complex such $0$-cells satisfy $e_0' = -e_0$, thus highlighting that $\G$ satisfies a zero Dirichlet boundary condition. The variable symmetry in condition (2) further implies that $\G(e,e_0) = 0$ for all $e \in \Mc_0^*$.

For a function in two variables, the notation $\pmb{d_j}$, $\pmb{\delta_j}$  and $\pmb{\Delta_j}$ is used to refer to, respectively, the differential operator, the co-differential operator and the discrete Laplace operator acting with respect to the $j$th variable. 

The main result of this section is as follows.
\begin{theorem}\label{thm:G}
\sloppy There exists a Green's function on the dual lattice manifold complex ${\G\,:\,\Mc_0^*\times \Mc_0^* \to \R}$ satisfying Definition \ref{def:G}, such that for any $e \in \Mc_1^*$ and $\tilde{e} \in \Mc_0^*$,
\begin{equation}\label{thm:G-1}
|\pmb{d_1}{\G}(e,\tilde{e})| \lesssim \left(1+(1+{\corr |e|_{\Wc}})\,{\corr D_{\Wc}(e,\tilde{e})}\right)^{-1},
\end{equation}
where $|\cdot|_{\corr \Wc}$ and $D_{\corr \Wc}(\cdot,\cdot)$ are functions measuring distances in the lattice manifold complex $\Wc^*$, as discussed in Section \ref{sec:distance}.

There further exists $\epsilon > 0$, independent of $\tilde{e}$, such that 
\begin{equation}\label{thm:G-2}
\|\pmb{d_1}\G(\cdot,\tilde{e})\|_{\infty} < \frac{1}{2} - \epsilon,
\end{equation}
with $\|\cdot\|_{\infty}$ defined in \eqref{inf-norm}.
\end{theorem}

The proof will be presented in Section \ref{sec:G-proof}. Note that it is  \eqref{thm:G-2} that ensures that $\pmb{d_1}\G$, through duality, gives rise to an equilibrium strain field around a dislocation.

Aside from the relevance of $\G$ for the problem of atomistic near-crack-tip plasticity, it is noted that, with $\G$ fully determined by the values it takes in $\Mc_0^{*,+} \times \Mc_0^{*,+}$ and, since $\chi(\Mc_0^{*,+}) = \Z^2$, the dual lattice manifold Green's function is equivalent to $G\,:\,\Z^2\times \Z^2 \to \R$ {\corr introduced in \eqref{G-normal-eqns}.} Thus results presented in Theorem \ref{thm:G} are also of independent interest. 
\subsection{Equilibrium dislocation configurations when $K=0$}\label{sec:existence_k0}
In this section dislocations in a cracked domain with no crack opening present will be considered.

The approach employed is motivated by the lattice complex duality which was used in \cite{H17} to prove existence and uniqueness, up to equivalence discussed in \eqref{equiv-rel}, of locally stable equilibria containing a given dislocation configuration in a finite homogeneous crystal, together with a precise representation of the associated bond-length 1-form.

The results presented here achieves a similar feat in the context of a lattice manifold complex representing a cracked lattice domain and heavily rely on Theorem \ref{thm:G}.

Prior to stating the result, it is recalled that in the definition of admissible dislocation configurations $\mathscr{B}_n$ given in \eqref{Bn} the constant  $\epsilon >0$ is fixed, together with a collection of $(b_i)^m_{i=1}$, where $b_i = \pm 1$ and $m \in \N$. Every $\mu \in \mathscr{B}_n$ admits a decomposition ${\mu = \sum_{i=1}^m b_i(\mathbb{1}_{e_i} + \mathbb{1}_{e_i'})}$, where each $e_i \in \Mc_2^{+}$ and each $e_i' \in \Mc_2^{-}$ is its crack reflection, as discussed in Section \ref{sec:latmancom}. Finally, recall that the notion of duality mapping $^*$ is discussed Section \ref{sec:dual_lmc}.
\begin{theorem}\label{thm:K0}
For all $n \in \N$ in the definition of $\mathscr{B}_n$ sufficiently large and for every 2-form $\mu \in \mathscr{B}_n$, there exists a corresponding locally stable equilibrium configuration $y_{\mu} \in \mathscr{V}_{\rm s}(\Mc_0)$ with $\alpha \in [\pmb{d}y_{\mu}]$ such that $\pmb{d}\alpha = \mu$ and $\pmb{\delta}\alpha = 0$. Furthermore, $\alpha^* = \pmb{d}^* \G_{\mu^*}$, with  $\G_{\mu^*} \in \Vsc_{\rm s}(\Mc_0^*)$ given by 
\begin{equation}\label{Gmustar}
\G_{\mu^*}(e) := \sum_{i=1}^m b_i \G(e,e_i^*),
\end{equation}
where $\G(\cdot,e_i^*)$ is the Green's function on the dual lattice manifold complex from Theorem~\ref{thm:G}.
\end{theorem}
The proof requires several steps with a detailed account given in  Section \ref{p-K0}. The essence of the argument rests on the observation that, subject to proving technical properties  \eqref{thm:G-1}, \eqref{thm:G-2} in  Theorem \ref{thm:G} and thanks to the the notion of duality being preserved, the strain fields around dislocations on the primal lattice manifold complex are equivalent to the strain fields around source points on its dual. 
\subsection{Near-crack-tip plasticity equilibrium configurations (the case $K > 0$)}\label{sec:existence_Knot0}
With the dislocations-only equilibria proven to exist and given an explicit characterisation in Section \ref{sec:existence_k0}, the more physically relevant case when an actual crack opening is present in the crystal will now be considered.

As was discussed in Section \ref{sec:cont_pred}, \cite[Theorem 2.3]{2018-antiplanecrack} asserts the existence of a locally stable equilibrium $y_{\rm c}$ in the case when $K >0$ is small enough and no dislocations are present (corresponding to setting $m$ in the definition of $\mathscr{B}_n$ in \eqref{Bn} to $m=0$). The proof of this result rests on the fact that a homogeneous crystal ($K=0$ and $m=0$) is a locally stable equilibrium exhibiting discrete ellipticity defined in \eqref{strong-s} and hence a standard application of the Implicit Function Theorem \cite{serge} yields existence of an equilibrium for $K$ sufficiently small. 

A similar approach that additionally takes correctly into account the issues around the differentiability of $\psi$ can be employed to prove the following. 
\begin{theorem}\label{thm:Knot0}
For all $n \in \N$ sufficiently large and for every $\mu \in \B_n$, there exists a locally unique and locally stable equilibrium  $y^{\rm c}_{\mu} \in \Vsc^{\rm crack}$ (space defined in \eqref{Vcrack}) containing the dislocation configuration $\mu$. In particular, it can be decomposed as 
\[
y^{\rm c}_{\mu}(K) = K\hat{u}_{\rm c} + y_{\mu} + \bar{u},
\]
where $\hat{u}_{\rm c}$ is defined in \eqref{uhatcmc}, $K>0$ is the stress intensity factor, $y_{\mu}$ is the locally stable equilibrium dislocation configuration whose existence was asserted in Theorem \ref{thm:K0} and ${\bar{u} \in \Hcc(\Mc_0)}$ is an atomistic correction. 
\end{theorem} 

\section{Conclusions {\corr and further research directions} }\label{sec:conclusions}
In this paper the framework of a lattice manifold complex was developed and shown to enable a mathematically rigorous study of near-crack-tip plasticity in the case of a square lattice under anti-plane kinematics. In particular, it extends the earlier work on the mathematics of atomistic plasticity presented in \cite{AO05,H17} and forms a basis for several avenues of further research to be discussed now.
\subsection*{Upscaling and the resulting mesoscopic study}  $\,$
Given the precise description of near-crack-plasticity equilibria provided in Section \ref{sec:existence_Knot0}, a rigorous study of spatial upscaling of the model is a clear future research direction.

The basic premise of such a study is to consider a rescaled lattice manifold complex, denoted by $\left(\frac{1}{n}\Mc\right)$, obtained by applying the procedure of Section \ref{sec:latmancom} to a rescaled lattice manifold $\frac{1}{n}\Mc_0$. It readily follows from the discussion near the definition of the set of admissible dislocation configurations $\mathscr{B}_n$ in \eqref{Bn}, that any dislocation configuration ${\mu_n \in \mathscr{B}_n}$ gives rise to a rescaled dislocation configuration $\tilde{\mu}_n \in \left(\frac{1}{n}\Mc\right)_2$, with an associated collection of dislocation cores $(x_1,\dots,x_m) \subset \R^2$, where necessarily $|x_i-x_j| \geq \epsilon$. Due to Theorem \ref{thm:Knot0}, one could then define
\[
\mathscr{E}_n(K,x_1,\dots,x_m):= \E(y^{\rm c}_{\mu_n}(K),K\hat{u}_{\rm c} + \hat{u}_{\mu_n}),
\]
where $\hat{u}_{\mu_n}$ is a suitably defined function capturing far-field behaviour due to dislocation configuration $\mu_n$, and subsequently study a suitable limit of $\mathscr{E}_n(K,x_1,\dots,x_m)$ as $n \to \infty$. 

In the light of corresponding results in the dislocations-only case \cite{ADLGP14,H17}, and given the persisting duality in the near-crack-tip plasticity setup, enabling a description of equilibrium strain fields through a Green's function on the dual lattice manifold complex, it seems clear that the limiting energy is of the form 
\begin{equation}\label{renom-energy}
\mathscr{E}(K,x_1,\dots,x_m) = K\sum_{i=1}^m \sqrt{|x_i|}\cos\left(\tfrac{\theta_{x_i}}{2}\right) + \sum_{i\neq j}b_i\,b_j G(x_i,x_j),
\end{equation}
where $\theta_{x_i}$ is the angle in the polar representation of $x_i$ and $G(x_i,x_j)$ is the (possibly rescaled) continuum Green's function in a crack geometry with zero Dirichlet boundary condition on the crack surface and $b_i$ is the Burgers vector associated with the $i$th dislocation core. Establishing a rigorous connection between $\mathscr{E}_n$ and $\mathscr{E}$ remains an open problem to be addressed. 

The renormalised energy in \eqref{renom-energy} is also a starting point for a potential mesoscopic study of dislocations, modelled as points $x_1,\dots,x_m$ in space, interacting with each other in the vicinity of a crack tip. A further upscaling regime to explore here is the many dislocations limit when $m \to \infty$, which is customarily handled in the framework of $\Gamma$-convergence. An energy form similar to \eqref{renom-energy} is for instance considered in a recent work in \cite{GvMPS20}. 

A particularly interesting modelling question to be explored is whether one can identify a regime in which the emergence of an experimentally verified {\it plastic-free zone} \cite{horton1982tem} can be established. In the many dislocations limit this would correspond to proving that an equilibrated dislocation density has support bounded away from the crack tip. 
\subsection*{Interplay between the magnitude of the stress intensity factor and the location and charge of the screw dislocations} $\,$
One of the key ideas of fracture mechanics is that there exists some critical stress intensity factor $K_c$, such that if $K > K_c$ then it is energetically favourable for the crack to propagate \cite{SJ12}. In atomistic fracture the picture is further complicated by the phenomenon of lattice trapping \cite{Thomson_1971}, which refers to there existing a range of values for $K$ for which the crack remains locally stable despite being above or below the critical $K_c$ -- this is discussed {\corr in a mathematically rigorous fashion} in \cite{2019-antiplanecrack}. 

Plastic deformations around the crack tip are known to influence the situation further and could contribute to shielding the material from further crack propagation \cite{majumdar1981crack,majumdar1983griffith,ZYLS10}. In the atomistic description this is related to the influence a dislocation has on the strain on bonds closest to the crack tip.  Depending on the sign of the stress intensity factor and the Burgers vector, it can either decrease this strain, shielding the material from further crack propagation, or increase the strain, thus assisting bond breaking. 

In the present work this is most pronounced in Theorem \ref{thm:G}, which by duality describes a strain field around a dislocation core. In particular, the strain on the atomistic bonds closest to the crack tip due a dislocation can be estimated by combining \eqref{thm:G-1} and \eqref{thm:G-2}. This paves the way to a future study of the interplay between the magnitude of the stress intensity factor and the location and charge of the screw dislocations.

At present the crack tip is fixed at the origin and the interactions across the crack manually removed and hence an open problem to be addressed here is to extend the framework to include a physically motivated and mathematically sound bond-breaking mechanism, thus allowing the crack tip to move.
\subsection*{A new approach for non-convex domains}$\,$
A cracked crystal considered in this study represents an extreme instance of a non-convex domain, with atoms at the crack surface not interacting with each other despite their physical proximity in the reference configuration, which{\corr , to some extent,} is representative of any non-convex domain with corners. The resulting issues are generally difficult to handle, leading to the exclusion of such domains from the analysis presented \cite{H17}.

The framework of the lattice manifold complex addresses this non-convexity by means of a suitable conformal mapping, which maps the square lattice onto a distorted half-space lattice. This suggests a general strategy of treating non-convex domains with corners, which centres around finding a suitable conformal mapping which straightens the corners. It would be interesting to explore the practical usefulness of such an approach, when applied to, for instance, a general non-convex lattice polygon.  
\subsection*{More general near-crack-tip plasticity setups}$\,$
The present study is restricted to a square lattice with anti-plane displacements and nearest neighbour atom interactions under a pair-potential function. As, outlined in \cite[Section 3.]{2018-antiplanecrack}, there remain significant technical obstacles, at present preventing going beyond this regime. They are mostly related to the phenomenon of surface relaxation, as studied in \cite{theilsurface11}, which is notoriously difficult to address mathematically. 

While the specific form of the pair potential $\psi$ defined in \eqref{psi} ensures an explicit construction of dislocations-only equilibria in Theorem \ref{thm:K0}, it remains feasible to study atomistic near-crack-tip plasticity under other reasonable pair potentials, as long as one derives an accurate far-field screw dislocation predictor for a cracked domain, which is possible using the complex square root mapping and a reflection argument. The resulting analysis would then proceed along the lines of \cite{2014-dislift}.

Extending the present work to such more general setups constitute a clear future research direction.
\section{Proofs}\label{sec:proofs}
\subsection{Proof of Theorem \ref{thm:G}}\label{sec:G-proof}
The proof will be presented in two separate subsections, one devoted to proving existence of such $\G$ and the estimate in \eqref{thm:G-1}, the other to establishing \eqref{thm:G-2}, which ensures that $\pmb{d_1}\G$ through duality gives rise to a strain field around a dislocation.

Throughout the proof it will often be useful to exploit the fact that $\G$ is equivalent to $G\,:\,\Z^2\times \Z^2 \to \R$ satisfying \eqref{G-normal-eqns}.

This equivalent description in particular ensures that many of the relevant results in \cite{2018-antiplanecrack} about the lattice Green's function are directly applicable, hence substantially shortening the resulting proof.

\subsubsection{Existence and decay estimates}
Similarly to the argument presented in \cite{2018-antiplanecrack}, the problem of establishing existence of such $\G$ can be tackled by employing a predictor-corrector approach, centered around decomposition $\G = \hat{\G} + \bar{\G}$, where $\hat{\G}$ is the corresponding continuum Green's function and $\bar{\G}$ is the atomistic correction lying in the Sobolev space $\dot{\mathscr{H}}^1(\Mc^*_0)$ in both variables.

With decay properties of the explicitly known $\hat{\G}$ readily calculable, this then paves the way to a technically involved boot-strapping argument establishing that the atomistic correction $\bar{\G}$ does not decay any slower than the predictor. 

The key difference between the current setup and the one discussed in \cite{2018-antiplanecrack} concerns the discrete boundary condition associated with the Green's function in question. As noted in Remark \ref{rem:dual-dirchlet}, the switch to the dual lattice manifold complex entails a change from a zero Neumann boundary condition to a zero Dirichlet boundary condition. This conceptual change can be seamlessly incorporated, allowing the proof to almost directly follow from the corresponding argument in \cite{2018-antiplanecrack}

In what follows the proof is laid out in several steps, quoting directly from the corresponding proofs in \cite{2018-antiplanecrack} where possible.

\hfill
\paragraph{Continuum crack Green's function with zero Dirichlet boundary condition}
The continuum crack Green's function corresponding the discrete problem at hand satisfies 
\begin{subequations}\label{hatGcont}
\begin{align}
-\Delta_x \hat{G}(x,s) &= \delta(x-s)&&\text{ for } x \in \R^2\setminus \Gamma_0,\label{hatGcont1}\\
\hat{G}(x_0,s) &= 0 &&\text{ for } x_0 \in \Gamma_0,\\
\hat{G}(x,s) &= \hat{G}(s,x) &&\text{ for }x,s \in \R^2,
\end{align}
\end{subequations}
where $\Delta_x$ refers to the continuum Laplace operator applied with respect to $x$. 

By reasoning as in \cite[Section 4.3.1]{2018-antiplanecrack}, an explicit solution is given by
\[
\hat{G}(x,s) = \frac{-1}{2\pi}\left(\log(|\omega(x)-\omega(s)| - \log(|\omega(x)-\omega(s)'|)\right),
\]
where $\omega$ is the complex square root mapping introduced in \eqref{om-map}, {\corr whereas in the second term} $\textstyle \omega(s)' = (-\omega(s)_1,\omega(s)_2)$. It is the minus sign in front of the second term that ensures that it satisfies the zero Dirichlet boundary condition. The following is further true, with the proof verbatim as in \cite[Lemma 4.4]{2018-antiplanecrack}.
\begin{lemma}\label{lem-hatG-decay}
For any $x,s\in \R^2\setminus\Gamma_0$ with $x\neq s$ and $\alpha \in \{1,\dots,4\}$,
\[
|\nabla^{\alpha}_x \hat{G}(x,s)| \lesssim (1+|\omega(x)|^{2\alpha-1}|\omega(x)-\omega(s)|)^{-1} + (1+|\omega(x)|^{\alpha}|\omega(x)-\omega(s)|^{\alpha})^{-1}
\]
and
\[
\medmath{|\nabla^{\alpha}_x\nabla_s \hat{G}(x,s)| \lesssim (1+|\omega(x)|^{2\alpha-1}|\omega(s)||\omega(x)-\omega(s)|^2)^{-1} + (1+|\omega(x)|^{\alpha}|\omega(s)||\omega(x)-\omega(s)|^{\alpha+1})^{-1}.}
\]
\end{lemma}

Using the convention described in Remark \ref{not-abuse2}, the dual lattice manifold complex equivalent can be defined as $\G\,:\, \Mc_0^* \times \Mc_0^* \to \R$ given by
\[
\hat{\G}(e,\tilde{e}):= \frac{-1}{2\pi}\left(\log({\corr D_{\Wc}(e,\tilde{e})}) - \log({\corr D_{\Wc}(e,\tilde{e}')}\right),
\]
if $e \neq \tilde{e}$ and $\hat{\G}(\tilde{e},\tilde{e}) = 0$. {\corr The  distance function $D_{\Wc}(\cdot,\cdot)$ will be defined in Section \ref{sec:distance} in the Appendix and it is recalled that $\tilde{e}'$ is the crack reflection of $\tilde{e}$, as introduced in Section~\ref{sec:latmancom} (see the Appendix for further details).} It can be readily checked that $\hat{\G}$ satisfies conditions (1) and (2) in Definition \eqref{def:G}, hence the atomistic correction will necessarily have to satisfy (1) and (2) by itself as well. 
\hfill\\

\paragraph{Predictor-corrector setup} 
The setup exactly mimics the one presented \cite[Section 4.3.1]{2018-antiplanecrack}.
For a fixed $\tilde{e} \in \Mc_0^*$ one considers
\[
\tilde{\E}_1(\mathcal{F}):= \left(\int_{\Mc_1^*}\frac{1}{2}\left(\pmb{d}_1 \hat{\G}(\cdot,\tilde{e}) + \pmb{d}\mathcal{F}\right)^2  - \frac{1}{2}\left(\pmb{d}_1 \hat{\G}(\cdot,\tilde{e})\right)^2\right) - 2\left(\mathcal{F}(\tilde{e}) + \mathcal{F}(\tilde{e}')\right)
\]
and for a fixed $e \in \Mc_0^*$, one considers 
\[
\tilde{\E}_2(\mathcal{F}):= \left(\int_{\Mc_1^*}\frac{1}{2}\left(\pmb{d}_2 \hat{\G}(e,\cdot) + \pmb{d}\mathcal{F}\right)^2  - \frac{1}{2}\left(\pmb{d}_2 \hat{\G}(e,\cdot)\right)^2\right) - 2\left(\mathcal{F}(e) + \mathcal{F}(e')\right).
\]
The following can be proven to hold for both functionals.
\begin{prop}\label{prop-E-G}
The energy difference functional $\tilde{\E}_i$ (where $i=1$ or $i=2$) is well-defined and smooth over $\Hcc(\Mc_0^*)$ defined in \eqref{Hcc-dual} and admits a unique minimiser $\bar{\G}_i \in \Hcc(\Mc_0^*)$.
\end{prop}
\begin{proof}
The first part of the proposition will follow directly from \cite[Proposition 4.5]{2018-antiplanecrack}, as long as it can be established that 
\[
\<{\tilde{\E}_1(0)}{\mathcal{F}} = \left(\int_{\Mc_1^*}\pmb{d}_1\hat{\G}(\cdot,\tilde{e})\pmb{d}\mathcal{F}\right) -2(\mathcal{F}(\tilde{e}) + \mathcal{F}(\tilde{e}'))
\]
\sloppy is a bounded linear functional on $\Hcc(\Mc_0^*)$, that is, if $\mathcal{F} \in \Hcc(\Mc_0^*)$, then ${|\<{\tilde{E}_1(0)}{\mathcal{F}}| \lesssim \|\pmb{d}\mathcal{F}\|_{\mathscr{L}^2(\Mc_1^*)}}$,  which requires a subtly different argument to accommodate the zero Dirichlet boundary condition. 

In the light of Remark \ref{rem:dual-dirchlet}, by definition $\mathcal{F} \in \Hcc(\Mc_0^*)$ satisfies the zero Dirichlet boundary condition, with $\mathcal{F}(e) = 0$ for any $e$ with $\chi(e) \subset \Gamma_0^{\Mc}$. Since $(\Mc_1^{*,+})' =\Mc_1^{*,-}$ (defined in \eqref{Mcp_pm_dual}), then in fact  
\[
\<{\tilde{\E}_1(0)}{\mathcal{F}} = 2\left(\int_{\Mc_1^{*,+}}\pmb{d}_1\hat{\G}(\cdot,\tilde{e})\pmb{d}\mathcal{F}\right) - 4\mathcal{F}(\tilde{e}).
\]
Setting, without loss of generality, $\tilde{e} \in \Mc_0^{*,+}$ and  $s := \chi(\tilde{e}) \in \Z^2$, it further follows that 
\[
\<{\tilde{\E}_1(0)}{\mathcal{F}} = 4\left(\sum_{m \in \Z^2}\sum_{\rho \in \Rc}\big(\hat{G}(m+\rho,s)-\hat{G}(m,s)\big)\big(f(m+\rho) - f(m)\big)- f(s)\right),
\]
where $\hat{G}$ satisfies \eqref{hatGcont} and $f\,:\,\Z^2 \to \R$ is the non-manifold equivalent of $\mathcal{F}$. This precisely pins down how the current setup can be reduced to the one considered in \cite{2018-antiplanecrack}. The remaining obstacle is the issue of the zero Dirichlet boundary condition. 

Exploiting \eqref{hatGcont}, it holds that $f(s) = \int_{\R^2}\nabla \hat{G}(x,s) \cdot \nabla I f(x) dx$, where $If$ is the P1 interpolation of $f$, akin to the one presented in \cite{2018-antiplanecrack}. Note that with the zero Dirichlet boundary condition there is no need to modify it near $\Gamma_0$. With $\nabla I(v)$ piecewise constant, in fact 
\begin{equation}\label{Umrho}
f(s) = \sum_{m \in \Z^2}\sum_{\rho \in \Rc} \left(\int_{U_{m,\rho}}\nabla_{\rho} \hat{G}(x,s)\,dx\right)\left(f(m+\rho) - f(m)\right),
\end{equation}
where $U_{m,\rho}$ is the union the two triangles sharing $(m,m+\rho)$ as an edge. By design if both $\tilde{m} \in \Gamma_0$ and $\tilde{m}+\tilde{\rho} \in \Gamma_0$, then $f(\tilde{m}+\tilde{\rho}) - f(\tilde{m}) = 0$, ensuring that the terms corresponding to $\Gamma_0$ vanish. The mid-point quadrature argument presented in \cite[Proposition 4.5]{2018-antiplanecrack} thus applies, confirming that $\delta\tilde{\E}_1(0)$ is indeed a bounder linear functional over $\Hcc(\Mc_0^*)$ and hence finishing the argument.

The fact that $\tilde{\E}_i$ admits a unique minimiser is due to the standard Lax-Milgram lemma, c.f. \cite[Lemma 4.6]{2018-antiplanecrack}. In particular $\bar{\G}_i$ satisfies
\begin{equation}\label{barG-eqn}
\<{
\delta \tilde{\E}_i(\bar{\G}_i)}{\mathcal{F}} = 0\;\text{ for all }\mathcal{F} \in \Hcc(\Mc_0^*),
\end{equation}
where, for a fixed $\tilde{e} \in \Mc_0^*$,
\[
\<{\delta \tilde{\E}_1(\bar{\G}_1)}{\mathcal{F}} = \left(\int_{\Mc_1^*}\left(\pmb{d}_1\hat{\G}(\cdot,\tilde{e})+ \pmb{d}\bar{\G}_1 \right)\pmb{d}\mathcal{F}\right) -2(\mathcal{F}(\tilde{e}) + \mathcal{F}(\tilde{e}')).
\]
and, for a fixed $e \in \Mc_0^*$,
\[
\<{\delta \tilde{\E}_2(\bar{\G}_2)}{\mathcal{F}} = \left(\int_{\Mc_1^*}\left(\pmb{d}_1\hat{\G}(e,\cdot)+ \pmb{d}\bar{\G_2} \right)\pmb{d}\mathcal{F}\right) -2(\mathcal{F}(\tilde{e}) + \mathcal{F}(\tilde{e}')).
\]
\end{proof}
\paragraph{Proof of Theorem \ref{thm:G}: existence of a Green's function}
It follows from applying \eqref{barG-eqn} in the case $i=1$ to $F = \mathbb{1}_{\tilde{e}} + \mathbb{1}_{\tilde{e}'}$ that $\hat{\G} + \bar{\G}_1$ satisfies conditions (1) and (3) in Definition \ref{def:G}. The remaining difficulty is to show that it also satisfies (2), which is equivalent to establishing that in fact $\bar{G}_1 = \bar{G}_2$. 

Firstly, it follows by construction that 
\begin{equation}\label{barg1-prop}
\text{ if }\,\chi(\tilde{e}) \subset \Gamma_0^{\Mc}\,\text{ then }\,\bar{\G}_1(e,\tilde{e}) = 0\,\text{ for all } \,e \in \Mc_0^*
\end{equation}
and likewise
\begin{equation}\label{barg1-prop2}
\text{ if }\,\chi(e) \subset \Gamma_0^{\Mc},\text{ then }\,\bar{\G}_2(e,\tilde{e}) = 0\,\text{ for all }\,\tilde{e} \in \Mc_0^*,
\end{equation}
which ensures that $\bar{\G}_1(e,\tilde{e}) = \bar{\G}_2(e,\tilde{e})$ if either $\chi(\tilde{e}) \in \Gamma_0^{\Mc}$ or $\chi(e) \in \Gamma_0^{\Mc}$.

Furthermore, the argument in \cite[Lemma 4.7]{2018-antiplanecrack} applies and thus 
\[
\pmb{d}_1\pmb{d}_2 \bar{\G}_1 = \pmb{d}_1\pmb{d}_2 \bar{\G}_2.
\]
This equality can only hold if, for all $e \in \Mc_1^*$ and $\tilde{e} \in \Mc_0^*$,
\[
\pmb{d}_1 \bar{\G}_1(e,\tilde{e})  = \pmb{d}_1 \bar{\G}_2(e,\tilde{e}) + \tilde{\mathcal{F}}_1(e),
\] 
for some function $\tilde{\mathcal{F}}_1\,:\, \Mc_1^* \to \R$. Setting $\tilde{e}$ to be such that $\chi(\tilde{e}) \subset \Gamma_0^{\Mc}$, \eqref{barg1-prop} and the fact that $\bar{\G}_2(e,\cdot) \in \Vsc_{\rm s}(\Mc_0^*)$ together imply that $\tilde{\mathcal{F}}_1$ is a zero function.

By the same reasoning, then, for all $e,\tilde{e} \in \Mc_0^*$,
\[
\bar{\G}_1(e,\tilde{e}) = \bar{\G}_2(e,\tilde{e}) + \tilde{\mathcal{F}_0}(\tilde{e}),
\]
for some function $\tilde{\mathcal{F}}_0 \,:\,\Mc_0^* \to \R$. Setting $e$ to be such that $\chi(e) \subset \Gamma_0^{\Mc}$, it follows from \eqref{barg1-prop2} and the fact that $\bar{\G}_1(\cdot,\tilde{e}) \in \Vsc_{\rm s}(\Mc_0^*)$ that  $\tilde{\mathcal{F}}_0$ is a zero function. Thus $\bar{\G}_1 = \bar{\G}_2 =: \bar{\G}$ and hence
\begin{equation}\label{hatG-barG}
\G: = \hat{\G} + \bar{\G}
\end{equation}
is a Green's function on the dual lattice manifold complex in the sense of Definition \ref{def:G}.

\hfill
\paragraph{Proof of Theorem \ref{thm:G}: Green's function decay properties}
With $\G$ given by \eqref{hatG-barG}, one can readily check that the estimate in \eqref{thm:G-1} holds for $\hat{\G}$. It will now be shown that $\pmb{d_1}\bar{\G}$ does not decay any slower. To this end, first a suboptimal result about the decay of the mixed derivative $\pmb{d_1}\pmb{d_2}{\G}$ has to be established.
\begin{lemma}
The Green's function on the dual lattice manifold complex $\G$ in \eqref{hatG-barG} satisfies, for any $\delta > 0$ and $e, \tilde{e} \in \Mc_1^*$,
\[
|\pmb{d}_1\pmb{d}_2 \G(e,\tilde{e})| \lesssim (1+(1+{\corr |e|_{\Wc}})(1+{\corr |\tilde{e}|_{\Wc}}) \,{\corr D_{\Wc}( e,\tilde{e})}^{2-\delta})^{-1} 
\]  
\end{lemma}
\begin{proof}
The proof follows from several straightforward adjustments of the technically involved boot-strapping argument presented in \cite[Section 4.3.3]{2018-antiplanecrack}.

In particular, the difference in the boundary condition is addressed by the action of the crack reflection symmetry operator on the dual lattice manifold complex, as introduced in Section \ref{sec:dual_lmc}.

By construction, it ensures that, $\pmb{\Delta_1} \G(e,\tilde{e}) = 0$ also when $\chi(e) \subset \Gamma_0^{\Mc}$, which is in contrast to $\Delta_d G(m_0,s) \neq 0$ for $G$ in \eqref{G-normal-eqns} when $m_0 \in \Gamma_0$. Thus locally, away from the origin, the homogenous lattice Green's function can be used to estimate $\bar{\G}$ and its derivatives in precisely the way it has been employed in \cite[Section 4.3.3]{2018-antiplanecrack}. The arbitrarily small $\delta > 0$ is an artefact of the boot-strapping argument saturating at the known rate of decay of $\hat{\G}$.

The fact that the terms of the form $(1+{\corr |e|_{\Wc}})^p$ for some negative power $p$ appear reflects the fact that if $e = [0,\tilde{\rho}] \in \Mc_1^*$ for some $\tilde{\rho} \in \Rc$ with $\tilde{\rho} \neq -e_2$, then
\[
\pmb{d_1}\pmb{d_2}\G(e,\tilde{e}) \sim \mathcal{O}({\corr |\tilde{e}|_{\Wc}})^{-3+\delta},
\]
which effectively follows from the fact that if $l:=\chi(\tilde{e}) \in \Z^2$ then 
\[
\left|\int_{U_{0,\tilde{\rho}}}\nabla_{\rho}\hat{G}(x,l)\,dx - (\hat{G}(\tilde{\rho},l) - \hat{G}(0,l))\right|  \sim \mathcal{O}(|\omega(l)|)^{-1},
\]
where $\hat{G}$ is the continuum Green's function satisfying \eqref{hatGcont} and $U_{0,\tilde{\rho}}$ is as in \eqref{Umrho}. 
\end{proof}
The mixed derivative  of $\G$ can in turn be used to prove the last result of this section, namely the sharp estimate of the decay of~$\pmb{d_1}{\G}$.
\begin{proof}[Proof of Theorem \ref{thm:G}: estimate in \eqref{thm:G-1}] 
It follows directly from Lemma \ref{lem-hatG-decay} that 
\[
|\pmb{d_1}\hat{\G}(e,\tilde{e})| \lesssim \left(1+(1+{\corr |e|_{\Wc}})\,{\corr D_{\Wc}(e,\tilde{e})}\right)^{-1},
\]
so it only remains to prove the same for $\bar{\G}$. Since $\bar{\G}(\cdot,\tilde{e}) \in \Hcc(\Mc_0^*)$, then it follows from \eqref{barG-eqn} that
\[
\pmb{d_1}\bar{\G}(e,\tilde{e}) = \frac{1}{4}\int_{\Mc_1^*}\pmb{d_1}\bar{\G}(\cdot,\tilde{e})\pmb{d_1}\pmb{d_2}\G(\cdot,e)
\]
Setting without loss of generality, as in Proposition \ref{prop-E-G}, $e \in \Mc_1^{*,+}$ and $\tilde{e} \in \Mc_0^{*,+}$ with $e = [l,l+\rho]$ for some $l \in \Z^2$ and $\rho \in \Rc$, and $s= \chi(\tilde{e}) \in \Z^2$, the mid-point quadrature argument from \cite[Proposition 4.5]{2018-antiplanecrack} and the decay estimates established in Lemma \ref{lem-hatG-decay} together ensure that in fact
\[
|\pmb{d_1}\bar{\G}(e,\tilde{e})| \lesssim I_1 + I_2,
\]
where $I_1 := \sum_{m \in \Z^2}h_1(m)$ and $I_2 := \sum_{m\in\Z^2}h_2(m)$, with  
\begin{align*}
\medmath{h_1(m)}\,&\medmath{:=}\,\medmath{\left(1+(1+|\omega(m)|^5)|\omega(m)-\omega(s)|\right)^{-1}\left(1+(1+|\omega(m)|)(1+|\omega(l)|)|\omega(m)-\omega(l)|^{2-\delta}\right)^{-1}},\\
\medmath{h_2(m)}\,&\medmath{:=}\,\medmath{\left(1+(1+|\omega(m)|)^3|\omega(m)-\omega(s)|^3\right)^{-1}\left(1+(1+|\omega(m)|)(1+|\omega(l)|)|\omega(m)-\omega(l)|^{2-\delta}\right)^{-1}.}
\end{align*}
Given that 
\[
h_1(l) \sim |\omega(l)|^5|\omega(l)-\omega(s)|,\quad h_1(s) \sim |\omega(l)||\omega(s)||\omega(l)-\omega(s)|^{2-\delta},
\]
\[
h_2(l) \sim |\omega(l)|^{-3}|\omega(l)-\omega(s)|^{-3},\quad h_2(s) \sim |\omega(l)|^{-1}|\omega(s)|^{-1}|\omega(l)-\omega(s)|^{-2+\delta},
\]
and with the tails decaying quickly, it readily follows that 
\[
I_1 + I_2 \lesssim (1+|\omega(l)||\omega(l)-\omega(s)|)^{-1},
\]
which is suboptimal, but sufficient to be able to conclude that
\[
|\pmb{d_1}\hat{\G}(e,\tilde{e})| \lesssim \left(1+(1+{\corr |e|_{\Wc}})\,{\corr D_{\Wc}(e,\tilde{e})}\right)^{-1},
\]
which finishes the proof.
\end{proof}

\subsubsection{Supremum bound}
The supremum bound is best proven by considering the non-manifold Green's function $G$ satisfying \eqref{G-normal-eqns}, which, as described below the statement of Theorem \ref{thm:G}, is equivalent to $\G$. 

To prove \eqref{thm:G-2} it is enough to show that there exists $\epsilon >0$ such that for any $m,s \in \Z^2$ and any $\rho \in \Rc$ (defined in \eqref{Rc}), it holds that 
\begin{equation}\label{G-normal-est1}
|G(m+\rho,s) - G(m,s)| \leq \frac{1}{2} -\epsilon.
\end{equation}

\sloppy Firstly, the obvious cases are handled. If both $m \in \Gamma_0$ and $m+\rho \in \Gamma_0$, then, by~\eqref{G-normal-eqn2}, ${G(m+\rho,s) - G(m,s) = 0 - 0 = 0}$. Similarly, if $s \in \Gamma_0$, then, applying first variable symmetry in \eqref{G-normal-eqn3} and then \eqref{G-normal-eqn2}, it follows that $G(m,s) = G(s,m) = 0$ for all $m \in \Z^2$, so again $G(m+\rho,s) - G(m,s) = 0$. 

To prove the remaining cases, $G$ is decomposed as
\begin{equation}\label{G-normal-dec}
G(m,s) = G^{\rm h}(m,s) + G^{\rm c}(m,s),
\end{equation}
where ${G^{\rm h}\,:\, \Z^2\times \Z^2 \to \R}$ is the full lattice Green's function satisfying, for any $m,s \in \Z^2$,
\begin{equation}\label{Gh}
G^{\rm h}(m,s) = G^{\rm h}(s,m),\quad -\Delta_{\rm d} G^{\rm h}(m,s) = \delta(m,s).
\end{equation}
It is a classic assertion based on the notion of semi-discrete Fourier transform \cite{SDFT} that there exists a Green's function satisfying \eqref{Gh} and further, for any $\rho \in \Rc$,  
\[
|G^{\rm h}(m+\rho,s) - G^{\rm h}(m,s)| \lesssim |m-s|^{-1}.
\] 
It is proven in \cite[Lemma 4.3]{H17} that, additionally,
\begin{equation}\label{G-normal-est2}
\sup_{\substack{m \in \Z^2,\\ \rho \in \Rc}} |G^{\rm h}(m+\rho,s) - G^{\rm h}(m,s)| = |G^{\rm h}(s+\tilde{\rho},s) - G(s,s)| = \frac{1}{4},
\end{equation}
where, due to rotational invariance, $\tilde{\rho} \in \Rc$ can be any of the four nearest neighbour directions. 

With both $G$ and $G^{\rm h}$ proven to exist, the function $G^{\rm c}\,:\,\Z^2 \times \Z^2 \to \R$ in \eqref{G-normal-dec} is thus defined via
\begin{equation}\label{Gc}
G^{\rm c}(m,s) := G(m,s) - G^{\rm h}(m,s)
\end{equation}
and it is straightforward to see that
\begin{subequations}
\begin{align}
-\Delta_{\rm d} G^{\rm c}(m,s) &= 0,\quad&&\text{ if } m \not\in\Gamma_0,\label{Gc-eqn1}\\
G^{\rm c}(m_0,s) &= -G^{\rm h}(m_0,s),\quad&&\text{ if } m_0 \in \Gamma_0,\label{Gc-eqn2}\\
G^{\rm c}(m,s) &= G^{\rm c}(s,m),\quad&&\text{ for all } m,s \in \Z^2.\label{Gc-eqn3}
\end{align}
\end{subequations}
In the light of remarks following \eqref{G-normal-est1} and the inequality in \eqref{G-normal-est2}, it thus remains to prove that there exists $\epsilon >0$ such that if $s \not\in \Gamma_0$ and either $m\not\in\Gamma_0$ or $m+\rho \not\in\Gamma_0$, then
\begin{equation}\label{G-normal-est3}
|G^{\rm c}(m+\rho,s)-G^{\rm c}(m,s)| \leq \frac{1}{4} - \epsilon.
\end{equation}
This can be achieved by invoking the Discrete Maximum Principle (\cite[Lemma 4.2.]{H17}), which can be adjusted to the present setting through defining 
\[
B_0^r:= \{ m \in \Z^2\;|\; |m| \leq r\,\text{ and } m \not\in\Gamma_0\},
\]
\begin{equation}\label{B0r-2}
\partial B_0^r:= \{m \in \Z^2\;|\; m \not\in B_0^r \text{ and } \exists \rho \in \Rc\,\text{ such that } m+\rho \in B_0^r\}\text{ and }\overline{B_0^r} := B_0^r \cup \partial B_0^r.
\end{equation}
The following lemma holds. 
\begin{lemma}[{\textbf{Discrete Maximum Principle}}]\label{DMP}
Suppose $u\,:\, \overline{B_0^r} \to \R$ is such that 
\[
\Delta_{\rm d} u(m) = 0,\quad \text{ for } m \in B_0^r. 
\]
Then
\[
\max_{m \in \overline{B_0^r}} |u(m)| = \max_{m \in \partial B_0^r}|u(m)|
\]
and unless the mapping
\[
m \in \partial B_0^r,\quad m \mapsto u(m)
\]
is constant, there exists $\epsilon >0$ such that, for any $m \in B_0^r$,
\[
|u(m)| < \max_{m \in \partial B_0^r}|u(m)| - \epsilon.
\]
\end{lemma}
\begin{proof}
The first part of the proof is exactly as in \cite[Lemma 4.2.]{H17}, with the second an immediate corollary, following from the fact that if there exists $\tilde{m} \in B_0^r$ such that
\[
|u(\tilde{m})| = \max_{m \in \partial B_0^r}|u(m)|
\]
then due to $u$ being harmonic, it follows that $u(m) = u(\tilde{m})$ for all $m \in \overline{B_0^r}$. This is impossible unless $u$ is constant on the boundary. If it is not, then $|u(m)|$ has to be strictly smaller than the maximum for every $m \in B_0^r$. The existence of $\epsilon$ then simply follows from the fact that $B_r^0$ is finite. 
\end{proof}
The inequality in \eqref{G-normal-est3} can now be proven in two lemmas.
\begin{lemma}\label{Gc-lemma1}
Let $\rho \in \Rc$, $s,s+\rho \not\in\Gamma_0$ and $s_0 \in \Gamma_0$ be such that 
\[
|s-\Gamma_0|:= \min\{|s-m_0|\;|\; m_0 \in \Gamma_0\} = |s-s_0|.
\] 
Further define $u^{\rm c}\,:\,\Z^2 \to \R$ by  $u^{\rm c}(m):= G^{\rm c}(m,s+\rho) - G^{\rm c}(m,s)$ where  $G^{\rm c}$ was introduced in \eqref{Gc}.

There exists $\epsilon >0$, such that for any $m\not\in\Gamma_0$, 
\begin{equation}\label{uc-bound}
|u^{\rm c}(m)| < |u^{\rm c}(s_0)| -\epsilon.
\end{equation}
Furthermore, 
\begin{enumerate}
\item if $\max\{|s-s_0|,|(s+\rho) - s_0|\} \leq \sqrt{2}$, then
\[
|u^{\rm c}(s_0)| = \frac{1}{\pi}-\frac{1}{4} \approx 0.06831. 
\]
\item If $\max\{|s-\Gamma_0|,|(s+\rho) - \Gamma_0|\} = 2$ and $\min\{|s-\Gamma_0|,|(s+\rho) - \Gamma_0|\} = 1$, then 
\[
|u^{\rm c}(s_0)| = \frac{3}{4} - \frac{2}{\pi} \approx 0.11338. 
\]  
\item If $\min\{|s-\Gamma_0|,|(s+\rho) - \Gamma_0|\} \geq 2$, then
\[
|u^{\rm c}(s_0)| < \frac{1}{\pi}-\frac{1}{4} \approx 0.06831.
\]
\end{enumerate}
\end{lemma}
\begin{proof}
Let 
\[
u(m):=G(m,s+\rho) - G(m,s), \quad u^{\rm h}(m) := G^{\rm h}(m,s+\rho) - G^{\rm h}(m,s),
\]
where $G$ is the Green's function satisfying \eqref{G-normal-eqns} and $G^{\rm h}$ is the full lattice Green's function satisfying \eqref{Gh}.

Due to the imposed spatial restriction on $s$, it holds that
\[
\Delta u^{\rm c}(m) = 0\;\text{ if } m \neq \Gamma_0\,\text{ and } u^{\rm c}(m_0) = -u^{\rm h}(m_0)\,\text{ if } m_0 \in \Gamma_0,
\]
which renders Lemma \ref{DMP} applicable to $u^{\rm c}$ restricted to $\overline{B_0^r}$ (defined in \eqref{B0r-2}). Letting $r$ go to infinity, together with  
\[
|u^{\rm c}(m)| = |u(m) - u^{\rm h}(m)|\leq (1+|\omega(s)||\omega(m)-\omega(s)|)^{-1} + |m-s|^{-1} \to 0,\text{ as } |m| \to \infty,
\]
implies  that in fact
\[
\max_{m \in \Z^2}|u^{\rm c}(m)| = \max_{m \in \Gamma_0}|u^{\rm c}(m)| = \max_{m \in \Gamma_0}|u^{\rm h}(m)| = |u^{\rm h}(s_0)|,
\] 
with the last equality following from the fact that $|u^{\rm h}(m)| \lesssim |m-s|^{-1}$ and $s_0 \in \Gamma_0$ being by definition closest to the source points $s$ and $s+\rho$. This establishes the inequality in \eqref{uc-bound}.

With the values of the full lattice Green's function $G^{\rm h}$ known explicitly (detailed e.g. in \cite{L-Green-sqr}) it is a simple matter of looking them up to to conclude the values $|u^{\rm h}(s_0)|$ in the three cases (1)-(3) considered, thus finishing the proof. 
\end{proof}
By variable symmetry in \eqref{Gc-eqn3}, Lemma \ref{Gc-lemma1} in fact establishes that if $m, m+\rho, s \not\in\Gamma_0$ then
\[
|G^{\rm c}(m+\rho),s) - G^{\rm c}(m,s)| \leq \frac{3}{4} - \frac{2}{\pi} < 0.12.
\]
It thus only remains to consider, without loss of generality, the case $m\not\in\Gamma_0$ but $m+\rho \in \Gamma_0$. 
\begin{lemma}
There exists $\epsilon > 0$ such that if $\rho \in \Rc$,  $s \not\in\Gamma_0$, $m\not\in\Gamma_0$ and $m+\rho \in \Gamma_0$, then
\[
|G^{\rm c}(m+\rho,s) - G^{\rm c}(m,s)| \leq \frac{1}{4} - \epsilon.
\]
\end{lemma}
\begin{proof}
With $m\not\in\Gamma_0$, \eqref{Gc-eqn1} implies that
\[
\sum_{\sigma \in \Rc} \left(G^{\rm c}(m+\sigma,s)- G^{\rm c}(m,s)\right) = 0,
\]
With $\rho \in \Rc$, it thus follows that
\[
|G^{\rm c}(m+\rho,s) - G^{\rm c}(m,s)| \leq \sum_{\substack{\sigma \in \Rc \\ \sigma \neq \rho}}\left|G^{\rm c}(m+\sigma,s) - G^{\rm c}(m,s)\right|
\]
Since $m+\rho \in\Gamma_0$, then $|m-\Gamma_0| = 1$ and necessarily there exists two distinct $\sigma_1,\sigma_2 \in \Rc$ such that $ 1 \leq |(m+\sigma_i) - \Gamma_0| \leq \sqrt{2}$ and also $\sigma_3 \in \Rc$ such that $|(m+\sigma_3)-\Gamma_0| = 2$. Thus, by exploting variable symmetry in \eqref{Gc-eqn3} and applying statements (1) \& (2) from Lemma \ref{Gc-lemma1} as well as \eqref{uc-bound}, it follows that  
\[
|G^{\rm c}(m+\rho,s) - G^{\rm c}(m,s)| \leq 2\left(\frac{1}{\pi} - \frac{1}{4}\right) + \left(\frac{3}{4}-\frac{2}{\pi}\right) - 3\epsilon = \frac{1}{4} - \tilde{\epsilon},
\]
where $\tilde{\epsilon} = 3\epsilon$. 
\end{proof}
This concludes the proof of the inequality in \eqref{thm:G-2}.
\begin{remark}
A similar albeit more involved decomposition-based argument can be applied directly to Green's function on the dual lattice manifold. The additional complications arise from the fact that the mapping 
\[
m_0\in  \Gamma_0,\quad  m_0 \mapsto G^{\rm h}(m_0,s)
\]
is not null, and hence neither is $G^{\rm c}$, thus the underlying requirement that these functions preserve crack reflection symmetry create additional issues near $\Gamma_0^{\Mc}$.   
\end{remark}

\subsection{Proof of Theorem \ref{thm:K0}}\label{p-K0}
Using the results about the Green's function from Section \ref{sec:G-proof}, Theorem \ref{thm:K0} can now be proven. The essence of the argument is as in \cite[Section 4.5]{H17}

Suppose $y_{\mu}$ from Theorem \ref{thm:K0} exists. As discussed after Definition \ref{def:stability}, any $\alpha \in [\pmb{d}y_{\mu}]$ necessarily satisfies $\|\alpha\|_{\infty} < \frac{1}{2}$, which ensures the set $[\pmb{d}y_{\mu}]$ admits only one function. If $\tilde{y} \in \Vsc_{\rm c}(\Mc_0)$ then for $t \in \R$ sufficiently small, $\|\alpha + t\pmb{d}\tilde{y}\|_{\infty} < \frac{1}{2}$ as well and hence
\[
\E(y_{\mu}+t\tilde{y},y_{\mu}) = \int_{\Mc_1} \left(\lambda t\, \alpha\, \pmb{d}\tilde{y} + \frac{\lambda t^2}{2}(\pmb{d}\tilde{y})^2 \right) \geq 0,
\]
which is only possible if 
\[
(\alpha,\pmb{d}\tilde{y}) = 0\quad\forall \tilde{y} \in \Vsc_{\rm c}(\Mc_0),
\] 
with the inner product introduced in \eqref{def-inner}. By the integration by parts formula in \eqref{def-int-parts}, it thus holds that $\pmb{\delta}\alpha  = 0$. Also, by definition $\pmb{d}\alpha = \mu$. 

By duality described in Section \eqref{sec:dual_lmc}, $\alpha \in \Vsc_{\rm s}(\Mc_1)$ gives rise to $\alpha^* \in \Vsc_{\rm s}(\Mc_1^*)$, which satisfies 
\begin{equation}\label{alphastar}
\pmb{\delta}^*\alpha^* = \mu^*\;\text{ and } \pmb{d}^*\alpha^* = 0,
\end{equation}
where $\mu^* \in \Vsc_{\rm s}(\Mc_0^*)$ is given by $\mu^* = \sum_{i=1}^m b_i(\mathbb{1}_{e_i^*} + \mathbb{1}_{(e_i^*)'})$. The second indicator function appears in this form due to \eqref{wc*-refl}.

It can be readily checked that $\alpha^* := \pmb{d}^* \G_{\mu^*}$ from \eqref{Gmustar} indeed satisfies \eqref{alphastar}. Since $\|\alpha\|_{\infty} = \|\alpha^*\|_{\infty}$, it thus remains to show that  $\|\pmb{d}^* \G_{\mu^*}\|_{\infty} < \frac{1}{2}$. 

By construction, if $e \in \Mc_1^*$ with $e \not\subset \Mc^{\pm}$, then $\pmb{d^*_1}\G(e,e^*_i) = 0$ for any $i \in \{1,\dots,m\}$.

On the other hand, for every $e \in \Mc_1^{*}$ such that $e \subset \Mc^{\pm}$, one can always find ${j \in \{1,\dots,m\}}$ for which 
\[
{\corr \max\{D(e,e^*_j),D_{\Wc}(e,e^*_j)\} = \min_i \max\{D(e,e^*_i),D_{\Wc}(e,e^*_i)\}.  }
\]
It follows from \eqref{thm:G-2} that there exists $\epsilon_0>0$ such that
\begin{equation}\label{G-eps0}
\left|\pmb{d^*_1}\G(e,e^*_j)\right| \leq \frac{1}{2}-(m-1)\epsilon_0,
\end{equation}
where $m \in \N$ is the number of dislocations in the definition of $\mathscr{B}_n$ in \eqref{Bn}.

{\corr As will be shown in Section~\ref{sec:min-sep-prop} in the Appendix, }the minimum separation {\corr property} in the definition of $\mathscr{B}_n$ implies that {\corr $e$ is far away from dislocation cores other than $e_j^*$, hence,} for $\epsilon_0$ in \eqref{G-eps0}, there exists $N \in \N$ such that if $n > N$, then
\begin{equation}\label{G-eps0-2}
|\pmb{d_1^*} \G(e,e_i^*)| < \epsilon_0\;\text{ for } i \neq j.
\end{equation}
Combining \eqref{G-eps0} and \eqref{G-eps0}, it thus follows that 
\[
\|\pmb{d^*}\G_{\mu^*}\|_{\infty} = \sup_{e \in \Mc_1^*}\left|\sum_{i=1}^m b_i \pmb{d_1^*}\G(e,e_i^*)  \right| < \frac{1}{2} - (m-1)\epsilon_0 + (m-1)\epsilon_0 = \frac{1}{2}.
\]

Finally, the fact that there exists $y_{\mu} \in \Vsc_{\rm s}(\Mc_0)$ such that $\alpha \in [\pmb{d}y_{\mu}]$ is a direct  consequence of the argument presented in \cite[Section 4.5]{H17}, which relies on the fact that the lattice manifold complex is path-connected and simply connected. 
\hfill  $\square$

\subsection{Proof of Theorem \ref{thm:Knot0}}\label{p-Knot0}
The proof of the theorem relies on Implicit Function Theorem, which is adapted from \cite[Theorem 22.27]{driver2003analysis}. 
\begin{theorem}\label{thm:IFT}
Suppose $X,Y,Z$ are Banach spaces, $U \subset X \times Y$ is an open set, ${f\,:\, U \to Z}$ is continuously Fr\'echet differentiable in both variables, $(x_0,y_0) \in U$ is such that ${f(x_0,y_0) = 0}$. If the  partial Fre\'echet derivative with respect to $x$ evaluated at $(x_0,y_0)$ and denoted by $D_x F(x_0,y_0) \in \mathcal{L}(X,Z)$ is an isomorphism from $X$ to $Z$, then there exists an open neighbourhood $X_0$ around $x_0$ and $Y_0$ around $y_0$, as well as a Fr\'echet differentiable function $g\,:\, Y_0 \to X_0$ such that for all $(x,y) \in X_0 \times Y_0$, 
\[
f(x,y) = 0 \iff x = g(y).
\]
\end{theorem}

The theorem about near-crack-tip plasticity equilibria can now be proven. 
\begin{proof}[Proof of Theorem \ref{thm:Knot0}]
Define 
\[
\mathscr{E}(u,K) := \E(K\hat{u}_{\rm c} + y_{\mu} + u, K\hat{u}_{\rm c} + y_{\mu}).
\]
This energy difference is clearly well-defined for any $K \in \R$ and $u \in \Vsc_{\rm c}(\Mc_0)$. 

However, if $K$ is small enough, then
\[
\|K\pmb{d}\hat{u}_{\rm c} + \alpha\|_{\infty} < \frac{1}{2},
\]
where $\alpha \in [\pmb{d}y_{\mu}]$, which implies that for $K$ small enough $\delta_u \mathscr{E}(0,K)$ is a bounded linear functional on $\Hcc(\Mc_0)$, since if ${v \in \Hcc(\Mc_0)}$, then 
\[
\<{\delta_u \mathscr{E}(0,K)}{v}= \int_{\Mc_1} \psi'(K \pmb{d}\hat{u}_{\rm c} + \pmb{d}y_{\mu})\pmb{d}v = \int_{\Mc_1} K\pmb{d}\hat{u}_{\rm c}\pmb{d}v + \alpha \pmb{d}v  = \int_{\Mc_1}K\pmb{d}\hat{u}_{\rm c}\pmb{d}v,
\]
where the last equality follows from the fact that $y_{\mu}$ is a locally stable equilibrium and the fact that $\Vsc_{\rm c}(\Mc_0)$ is dense in $\Hcc(\Mc_0)$, which follows directly from \cite[Proposition 9.]{2012-lattint}.

The argument presented in the proof of \cite[Theorem 2.3]{2018-antiplanecrack}, based on the fact that the non-manifold version of $\hat{u}_{\rm c}$ satisfies the PDE in \eqref{PDE-c}, can thus be applied, ensuring that $|\<{\delta_u \mathscr{E}(0,K)}{v}| \lesssim \|\pmb{d}v\|_{\mathscr{L}^2(\Mc_0)}$. It hence follows from \cite[Lemma 4.1]{2013-disl} that $\mathscr{E}$ can be uniquely extended to $\Hcc(\Mc_0) \times I_{\delta}$ where $I_{\delta} := \{ K \in \R\;|\; -\delta < K < \delta\}$, for some $\delta$ small enough. 

If the domain of $\mathscr{E}$ is further restricted to an open subset
\[
U_{\delta} := \{(u,K) \in \Hcc(\Mc_0) \times I_{\delta}\;|\; \|\pmb{d}u\|_{\mathscr{L}^2(\Mc_0)} < \delta \},
\]
and $\delta$ taken to be sufficiently small, then for any $(u,K) \in U_{\delta}$ it holds that 
\[
\|K\pmb{d}\hat{u}_{\rm c} + \pmb{d}u + \alpha \|_{\infty} < \frac{1}{2},
\]
from which it is straightforward to conclude that $\mathscr{E} \in C^2(U_{\delta},\R)$. 

Furthermore, if $K_0 = 0$ and $\bar{u}_0 := 0 \in \Hcc(\Mc_0)$,  then by virtue of Theorem \ref{thm:K0} asserting that $y_{\mu}$ is a locally stable equilibrium, it follows that $\delta_u \mathscr{E}(\bar{u}_0,K_0) = 0$ and further
\begin{equation}\label{d2E-stab}
\<{\delta^2_u \mathscr{E}(\bar{u}_0,K_0)v}{v} = \lambda \|\pmb{d}v\|^2_{\mathscr{L}^2(\Mc_0)}, 
\end{equation}
where $\lambda > 0$ is the constant in the definition $\psi$ in \eqref{psi}. 

To invoke Theorem \ref{thm:IFT} (Implicit Function Theorem), it remains to set 
\[
f:= \delta_u \mathscr{E}\,:\, U_{\delta} \to (\Hcc(\Mc_0))'
\]
and observe that 
\[
f(\bar{u}_0,K_0) = 0\,\text{ and } D_{u}f(\bar{u}_0,K_0) = \delta^2_u \mathscr{E}(\bar{u}_0,K_0),
\]
with $D_u f(\bar{u}_0,K_0)$ an isomorphism due to \eqref{d2E-stab}. 

With all the conditions of the Implicit Function Theorem satisfied, it thus follows that there exists a locally unique path $\{(\bar{u}(K),K) \in \Hcc(\Mc_0) \times \R\;|\; K \in (-\delta_0,\delta_0)\}$, for some $\delta_0 >0$ small enough, such that $\delta_u \mathscr{E}(\bar{u}(K),K) = 0$.    

The fact that $y_{\mu}^{\rm c}(K) = K\hat{u}_{\rm c} + y_{\mu} + \bar{u}(K) \in \Vsc_{\rm s}(\Mc_0)$ is a locally stable equilibrium in the sense of Definition \ref{def:stability} follows naturally from the fact that $\delta_u \mathscr{E}(\bar{u}(K),K) = 0$ and $\|\pmb{d}\tilde{y}\|_{\infty} \leq \|\pmb{d}\tilde{y}\|_{\mathscr{L}^2(\Mc_0)}$.

Furthermore, it can be readily checked by inspecting the definition of $\hat{u}_{\rm c}$ and recalling that if $v \in \Hcc(\Mc_0)$ then $|v(e)| \lesssim \log|e|$, which follows directly from \cite[Lemma 4.3]{2019-antiplanecrack}, that the far-field behaviour of $\hat{u}_{\rm c}$ is sufficiently dominant to conclude that if $K > 0$, then $y_{\mu}^{\rm c}(K) \in \Vsc^{\rm crack}$, which concludes the proof.
\end{proof}

\vspace{1cm}
\paragraph{Acknowledgements} The author would like to thank Tom Hudson, Christoph Ortner, Patrick van Meurs, Masato Kimura and Angela Mihai for their continuing support throughout the work on this manuscript. 

\vspace{1cm}
\appendix
\section{Orientation, duality, crack reflection symmetry and measuring distances}\label{app1}
In this section the notions of orientation, duality and crack reflection symmetry, as visually introduced through Figure \ref{fig:3}, Figure \ref{fig:9} and \ref{fig:6} will now be introduced rigorously. 

This is then followed by a discussion about measuring distances in the lattice manifold complex framework.  
\subsection{Primal lattice manifold complex}
A $0$-cell $a \in \Wc_0$ is uniquely represented by 
\begin{equation}\label{wc0}
a = (x,y,o) \in \R^2 \times \{+1,-1\},
\end{equation}
that is by its position in $\Wc = \R^2$ and its orientation. The reversely oriented $-a$ is thus represented by $-a = (x,y,-o)$ and $\chi(a) = \chi(-a) = (x,y) \in \R^2$.

The orientation of a $0$-cell, though seemingly artificial, is in fact crucial for the notion of duality discussed in Section \ref{sec:dual_lmc} and also for the boundary operator $\partial$ discussed in Section \ref{sec:latmancom}, which dictates the following description of $1$-cells. 

Given $e \in \Wc_1$, there exists a unique collection of two positively oriented $0$-cells $a_1,a_2 \in \Wc_0$ such that $\partial e = -a_1 \cup a_2$, thus providing $1$-cells with concise notation in terms positively oriented $0$-cells with $e = [a_1,a_2]$ and $-e = [a_2,a_1]$. Since $e$ can be interpreted as an arrow with initial point $a_1$ and terminal point $a_2$, it is thus not needed to explicitly specify whether a given $1$-cell is positively oriented. Furthermore, $\chi([a_1,a_2]) = \chi([a_2,a_1])$ is the open curved line in $\R^2$, not containing $\chi(a_1)$ and $\chi(a_2)$.

Likewise, given $E \in \Wc_2$ (excluding the special central $2$-cell, which can be handled similarly) there exists a unique (up to permutations and orientation reversal) collection of $1$-cells $e_1,\dots,e_4 \in \Wc_1$ such that $\partial E = e_1 \cup e_2 \cup e_3 \cup e_4$ and further there exists a unique (up to permutations) collection of positively oriented $a_1,\dots,a_4$ such that 
\[
e_1 = [a_1,a_2], e_2 = [a_2,a_3], e_3 = [a_3,a_4], e_4 = [a_4,a_1],
\]
 i.e. a loop consisting of four $0$-cells.   Thus positively oriented $0$-cells provide a convenient notation for $2$-cells through $E = [a_1,a_2,a_3,a_4]$. It can be also easily seen that $-E = [a_1,a_4,a_3,a_2]$, thus to avoid ambiguity, a convention is adopted that a $2$-cell $E \in \Wc_2$ is positively oriented if it is a {\it clockwise }loop. Similarly, $\chi(E) = \chi(-E)$ is the open subset of $\R^2$ bounded by curved lines $\chi([a_1,a_2])$, $\chi([a_2,a_3])$, $\chi([a_3,a_4])$, $\chi([a_4,a_1])$. Note that a similar analysis applies to the central $2$-cell $E_0$, but with eight positively oriented $0$-cells instead of four.

The notion of crack reflection symmetry will now be explicitly introduced, assigning to each $e \in \Wc_p$ its crack reflection $e' \in \Wc_p$ as follows. Figure \ref{fig::2} provides a useful illustration.

For $a \in \Wc_0$ with $a = (x,y,o) \in \R^2\times \{+1,-1\}$ as in \eqref{wc0}, its reflection is chosen to be defined as 
\begin{equation}\label{wc0-refl}
a' := (-x,y,o),
\end{equation}
thus a positively (resp. negatively) oriented $0$-cell is mapped to a positively (resp. negatively) oriented $0$-cell. 

Likewise for $e \in \Wc_1$ given by $e = [a_1,a_2]$ and $E \in \Wc_2$ by $E = [a_1,a_2,a_3,a_4]$, 
\begin{equation}\label{wc12-refl}
e' := [a_1',a_2'], \quad E' := [a_1',a_2',a_3',a_4'],
\end{equation}
Note that for any $e \in \Wc_p$, its reflection $e'$ is unique and $(e')' = e$. In particular, if $E$ is positively oriented, then $E'$ is negatively oriented. 

The inverse of the isomorphism $\omega_{\Mc}$ defined in \eqref{om_mc} can be extended to apply to oriented $p$-cells by simply saying that it preserves orientation, thus for any $e \in {\corr \Wc_p}$ there exists a unique $\omega_{\Mc}^{-1}(e) \in {\corr \Mc_p}$ {\corr and in particular } ${\corr \omega_{\Mc}^{-1}(-e) = -\omega_{\Mc}^{-1}(e)}$. 

This ensures that both the explicit representations of oriented $p$-cells and the notion of crack reflection symmetry applies to the $\Mc$-manifold description, as illustrated in Figure~\ref{fig::1}. In particular it is noted that any $a \in \Mc_0$ can be uniquely represented by 
\begin{equation}\label{mc0-a}
a = ((x,y),b,o) \in \R^2 \times \{+1,-1\} \times \{+1,-1\},
\end{equation}
with $b$ denoting the branch of the manifold and $o$ the orientation. If $e \in \Mc_p$ then $\chi(e)$ is defined as $\chi(e):= \omega_{\Mc}^{-1} \left(\chi(\omega_{\Mc}(e))\right)$, so in particular $\chi(a) = ((x,y),b)$. 
 
Similarly, a crack reflection of $e \in \Mc_p$ is defined as
\[
e' := \omega_{\Mc}^{-1}\left(\omega_{\Mc}(e)'\right),
\]
which in particular implies that for $a$ in \eqref{mc0-a}, $a' = ((x,-y),-b,o)$. 

Crucially this indeed ensures that $(\Mc_p^+)' = \Mc_p^-$ (defined in \eqref{Mcp_pm}) and $(\Wc_p^+)' = \Wc_p^-$ (defined in \eqref{Wcp_pm}).
\subsection{Dual lattice manifold complex}
In this section, the notion of orientation and the action of the isomorphism $^*$ from \eqref{iso} (visually described in Figure \ref{fig:9}) will be rigorously defined.

If $E \in \Wc_2$ is positively ({\it clockwise}) oriented then $E^* \in \Wc_0^*$ is a positively oriented $0$-cell located within the loop that $E$ represents. If $e \in \Wc_1$, then $e^* \in \Wc_1^*$ is obtained by a {\it right-hand} turn in the arrow interpretation. Finally, a positively oriented $0$-cell $a \in \Wc_0$ is mapped to a positively ({\it anti-clockwise}) oriented $2$-cell $a^* \in \Wc_2^*$. A visual depiction of duality is provided in Figure \ref{fig::7}. 

Note that on the original lattice manifold complex a $2$-cell is described as positively oriented if it has clockwise orientation, whereas on the dual lattice manifold an anti-clockwise orientation of a $2$-cell is considered positive orientation. This change ensures that \eqref{def-bond-op-dual} holds true. 

With orientation and duality introduced, the notion of crack reflection symmetry enforced through \eqref{wc*-refl}, showcased in Figure \ref{fig:6}, can now be discussed. 

In the $\Wc^*$-manifold description, as shown in Figure \ref{fig::5}, if, as discussed in \eqref{wc0}, $a^* \in \Wc_0^*$ is given by 
\[
a^* = (x^*,y^*,b^*) \in \R^2 \times \{+1,-1\},
\]
then the crack reflection from \eqref{wc*-refl} implies that $(a^*)' = (-x^*,y^*,-b^*)$, i.e., unlike for the $0$-cells of the original complex, the orientation is swapped and now $-(a^*)'$ is the positively orientated $0$-cell. As a result, a positively oriented $0$-cell description of $1$-cells and $2$-cells captures reflections on the dual complex as follows. 

For $e^* \in \Wc_1^*$ with $e^* = [a_1^*,a_2^*]$ and $E^* \in \Wc_2^*$ with $E^* = [a_1^*,a_2^*,a_3^*,a_4^*]$, the reflections are given by 
\[
(e^*)' = [-(a_2^*)',-(a_1^*)'],\quad (E^*)' = [-(a_2^*)', -(a_1^*)',-(a_4^*)',-(a_3^*)],
\]
in particular with $(E^*)'$ remaining positively (anticlockwise) orientated. 

As in the case of the primal lattice manifold complex,  the mapping $\omega_{\Mc}$ from \eqref{om_mc} ensures that this discussion also applies to the $\Mc^*$-manifold description, as shown in Figure \ref{fig::4}.

\subsection{Measuring distances}\label{sec:distance}
With the lattice manifold complex admitting two isomorphic descriptions through manifolds $\Wc$ and $\Mc$, it is important to introduce several relevant notions of a distance, with care taken to account both for the geometry and for the structure of the complex.

If $x,y, 0 \in \R^2$ then $D(x,y)$ is the Euclidean distance (note the capital letter to distinguish it from the differential operator $\pmb{d}$)  and $|x|:=D(x,0)$ is the Euclidean norm. The shortest distance between  $A \subset \R^2$ and $B \subset \R^2$ is 
\begin{equation}\label{set-dist}
D(A,B):= \inf \{\,D(x,y)\;|\; x \in A,\,y \in B\}.
\end{equation}

With {\it the space} $\Wc$ coinciding with $\R^2$,  these definitions thus apply to elements of $\Wc$. When measuring distances, the additional notion of orientation of the $p$-cells in {\it the complex} $\Wc$ (or its dual $\Wc^*$) is disregarded with the help of the function $\chi$ introduced in Section~\ref{sec:latmancom}, noting that for any $e \in \Wc_p$ it holds that $\chi(e) \subset \R^2$. Thus, if further $\tilde{e} \in \Wc_q$, the distances between cells in the complex and its dual are defined as 
\[
D(e,\tilde{e}) := D(\chi(e),\chi(\tilde{e}))\,\text{ and } D(e^*,\tilde{e}) := D(\chi(e^*),\chi(\tilde{e})),
\]
with the shortest distance between two subsets of $\Wc = \R^2$ defined in \eqref{set-dist}. Likewise, it will be useful to distinguish $|e| := D(e,E^*_0)$, where $E_0 \in \Wc_2$ is the central $2$-cell. 

Now let $k,l \in \Mc$ with $k = (k_x,k_b)$ and $l = (l_x,l_b)$. The basic notion of a distance on $\Mc$ used throughout disregards the branch indicators $k_b, l_b \in \{\pm 1\}$ and is given by { \corr
\begin{equation}\label{D-Mc-1}
D(k,l):= D(k_x,l_x),
\end{equation}}
i.e. the Euclidean distance between $k_x, l_x \in \R^2$. This gives rise to a semi-norm $|k|:= |k_x|$. 

{\corr An alternative relevant notion of a distance on $\Mc$ that is suitable for discussing the distance between dislocations cores is
\begin{equation}\label{D-Mc-2}
D_{\Wc}(k,l) := D(\omega_{\Mc}(k),\omega_{\Mc}(l)),
\end{equation}
with the associated (semi-) norm $|k|_{\Wc} := |\omega_{\Mc}(k)| = |k|^{1/2}$. It can be verified that the following useful identity holds, 
\begin{equation}\label{c-sqrt-id}
D(k,l) = D_{\Wc}(k,l)\;\tilde{D}_{\Wc}(k,l),
\end{equation}
where $\tilde{D}_{\Wc}(k,l) := D(\omega_{\Mc}(k),-\omega_{\Mc}(l))$.

The notion of orientation of $p$-cells in $\Mc$ is again disregarded {\corr as far as distances are concerned} and thus, for any $e \in \Mc_p$ (or $e \in \Mc_p^*$) and $\tilde{e} \in \Mc_q$, 
\[
D(e,\tilde{e}) := \inf \{\,D(k,l)\;|\; k \in \chi(e),\,l \in \chi(\tilde{e}) \},
\]
with same holding for $D_{\Wc}(e,\tilde{e})$. Likewise,
\[
|e|:= D(e,\omega_{\Mc}^{-1}(E_0^*)),
\]
is a semi-norm measuring the distance of a $p$-cell from the origin, with a similar definition for $|e|_{\Wc}$. 
}
\subsubsection{\corr {Minimum separation property}}\label{sec:min-sep-prop}
{\corr Consider two dislocation cores represented as positively oriented $e,\tilde{e} \in \Mc_2^+$. The origin of the minimum separation property stemming from the definition of $\mathscr{B}_n$ in \eqref{Bn}, concerns the idea that the strain field associated with the core at $e$ should be negligible near the core $\tilde{e}$. As established in Theorem \ref{thm:K0}, the strain fields around dislocations on the primal lattice manifold complex are equivalent to the strain fields around source points on its dual, thus it is equivalent to require that there exists (small) $\epsilon_0 > 0$ such that  
\begin{equation}\label{eq-min-sep-1}
|\pmb{d_1^*} \G(e^*,\tilde{e}^*)| < \epsilon_0,
\end{equation}
where $\G$ is the Green's function on the dual lattice manifold from Theorem \ref{thm:G}. This is used in the proof of Theorem \ref{thm:K0} in \eqref{G-eps0-2}. 

The decay estimate of $\G$ from \eqref{thm:G-1} implies that \eqref{eq-min-sep-1} can be restated as requiring that there exists (large enough) $\delta_1 > 0$, such that 
\begin{equation}\label{eq-min-sep-2}
|e|_{\Wc}\,D_{\Wc}(e,\tilde{e}) > \delta_1.
\end{equation}

It is more natural, however, to phrase the notion of a minimum separation between $e$ and $\tilde{e}$ in terms the distances on $\Mc$ introduced in \eqref{D-Mc-1} and \eqref{D-Mc-2}.  It turns out, however, that in this context, it is insufficient to  consider either $D(\cdot,\cdot)$ or $D_{\Wc}(\cdot,\cdot)$ alone, due to the following two examples. 

Firstly, consider two sequences of positively oriented $\{e_k\},\{\tilde{e}_k\} \subset \Mc_0^{*,+}$ with ${\chi(e_k) = (-k,1)}$ and ${\chi(\tilde{e}_k) = (-k,-1)}$. It can be readily established that 
\[
D(e_k,\tilde{e}_k) = 2\quad \forall\, k,\quad\text{ whereas }\quad D_{\Wc}(e_k,\tilde{e}_k) \to \infty\;\text{ as }\;k \to \infty.
\]
This implies that only the distance $D_{\Wc}(\cdot,\cdot)$ correctly identifies that the cells are separated by the crack surface.

On the other hand, for two sequences of positively oriented $\{a_k\},\{\tilde{a}_k\} \subset \Mc_0^{*,+}$ with ${\chi(a_k) = (k,\delta)}$ and ${\chi(\tilde{a}_k) = (k,-\delta)}$ where $\delta >0$ is arbitrarily large, it holds, for all $\delta >0$, that
\[
D(a_k,\tilde{a}_k) = 2\delta,\quad\text{ whereas }\quad D_{\Wc}(a_k,\tilde{a}_k) \to 0\;\text{ as }\;k \to \infty.
\]
Thus, this time, it is $D(\cdot,\cdot)$ that correctly identifies the separation between the cores.

To rectify this, the notion of minimum separation in \eqref{Bn} is phrased as a maximum over $D(\cdot,\cdot)$ and $D_{\Wc}(\cdot,\cdot)$ and to fit the phrasing in \eqref{eq-min-sep-2}, it can be restated as requiring that there exists a (large enough) $\delta_2 > 0$ such that
\begin{equation}\label{eq-min-sep-3}
\max\{D(e,\tilde{e}),D_{\Wc}(e,\tilde{e})\} > \delta_2.
\end{equation}

It thus remains to show that \eqref{eq-min-sep-3} implies \eqref{eq-min-sep-2}, which will ensure \eqref{G-eps0-2} holds true, as necessary to complete the proof of Theorem~\ref{thm:K0}.
\begin{lemma}
For any large enough $\delta_1 > 0$, there exists $\delta_2 > 0$ such that, for any $e,\tilde{e} \in \Mc_0^+$,
\[
\Big( \max\{D(e,\tilde{e}),D_{\Wc}(e,\tilde{e})\} > \delta_2 \Big) \implies \Big( |e|_{\Wc}\, D_{\Wc}(e,\tilde{e}) > \delta_1\Big).
\]
\end{lemma}
\begin{proof}
The first observation is that since $e,\tilde{e} \subset \Mc^+$, neither is equal to the special $2$-cell $\omega_{\Mc}^{-1}(E_0)$ with eight $1$-cells as its boundary and thus there is a constant $c_0 > 0$ such that  $|e| > c_0$ and $|\tilde{e}| > c_0$.

Secondly, for the result to apply to the maximum over two notions of a distance, it has to apply to each separately. The case when one seeks to find $\delta_2$ such that $D_{\Wc}(e,\tilde{e}) > \delta_2$ is immediate -- 
given $\delta_1 > 0$, it suffices to set $\delta_2 = \frac{\delta_1}{\sqrt{c_0}}$, since
\[
|e|_{\Wc}\,D_{\Wc}(e,\tilde{e}) > \sqrt{c_0}\,D_{\Wc}(e,\tilde{e}) > \sqrt{c_0} \delta_2 = \delta_1,
\]
where the first inequality is due to the first observation at the start of this proof.

The other case, in which one seeks an appropriate $\delta_2 > 0$ such that $D(e,\tilde{e}) > \delta_2$, will be dealt with in three separate sub-cases, with each argument based on the identity from \eqref{c-sqrt-id}.
\paragraph{Case 1: $|e| \geq |\tilde{e}|$}$\,$\\ In this case it suffices to set $\delta_2 = 2 \delta_1$, since $|e| \geq |\tilde{e}| \implies |e|_{\Wc} \geq |\tilde{e}|_{\Wc}$ and thus
\[
2|e|_{\Wc}D_{\Wc}(e,\tilde{e}) \geq (|e|_{\Wc} + |\tilde{e}|_{\Wc})D_{\Wc}(e,\tilde{e}) \geq \tilde{D}_{\Wc}(e,\tilde{e})\,D_{\Wc}(e,\tilde{e}) = D(e,\tilde{e}) > \delta_2,
\]
where the first inequality relies on the case considered, the second follows from triangle inequality and the subsequent equality follows from the identity in \eqref{c-sqrt-id}.

\paragraph{Case 2: $|e| \ll |\tilde{e}|$}$\,$\\ 
The requirement that $|\tilde{e}| \gg |e|$ can be restated as $|e| < c_1$ for some constant $c_1 > 0$. It then holds that $\sqrt{c_0} < |e|_{\Wc} < \sqrt{c_1}$ and $|\tilde{e}| > \sqrt{\delta_2 - c_0}$, which leads to
\begin{equation}\label{eq-min-sep-4}
|e|_{\Wc}\,D_{\Wc}(e,\tilde{e}) > \sqrt{c_0}\left(\sqrt{\delta_2 - c_0} - \sqrt{c_1}\right).
\end{equation}
Setting the right-hand side equal to $\delta_1$ and solving for $\delta_2$, one gets
\[
\delta_2 = \left(\frac{\delta_1}{\sqrt{c_0}} + \sqrt{c_1}\right)^2 + c_0.
\]
Note that here regardless of the size of $c_1$, the lemma is to hold for $\delta_2$ large enough, so it suffices to consider $\delta_2 \gg c_1$, in which case $\delta_1 \gg c_1$, so the right-hand side in \eqref{eq-min-sep-4} is positive. 

\paragraph{Case 3: $|e| < |\tilde{e}|$ but $|e| \not\ll |\tilde{e}|$}$\,$\\
To cover this case it suffices to consider $D(e,\tilde{e}) = \tilde{\delta}_2 > \delta_2$, $|e| < |\tilde{e}|$ and $|e| \geq c_1$. This final case covers the other distinct possibility that both $|e| \gg 1$ and $|\tilde{e}| \gg 1$, but it is only true that $|\tilde{e}| > |e|$. In other words, one first fixes $\tilde{\delta}_2$ and then considers $c_1 \gg \tilde{\delta_2}$. It follows directly from the definition of the complex square root mapping that in this case, for an arbitrarily small $c_2 > 0$ one can find $c_1$ large enough such that
\[
1-c_2 < \frac{|e|_{\Wc}}{|\tilde{e}|_{\Wc}} < 1 + c_2.
\]
One can use it to conclude that
\[
|e|_{\Wc}\,D_{\Wc}(e,\tilde{e}) = \frac{|e|_{\Wc}\,D(e,\tilde{e})}{\tilde{D}_{\Wc}(e,\tilde{e})} \geq \frac{\delta_2}{2} \frac{|e|_{\Wc}}{|\tilde{e}|_{\Wc}} \geq \frac{\delta_2}{2}(1 - c_2).
\]
Setting the right-hand side equal to $\delta_1$ one obtains $\delta_2 = \frac{2\delta_1}{1-c_2}$, so again a large enough $\delta_1$ ensures that $\delta_2$ is also large enough. 
\end{proof}
}

\bibliographystyle{myalpha} 


\newcommand{\etalchar}[1]{$^{#1}$}
 \newcommand{\noop}[1]{}

\end{document}